%% file: ex_article.tex
\documentclass[hidelinks,onefignum,onetabnum]{siamart251216}

\usepackage{mathtools}
\usepackage{booktabs}
\usepackage{algorithm}
\usepackage{algorithmic}
\usepackage[ruled, linesnumbered, vlined, resetcount, algo2e]{algorithm2e}
\usepackage{subcaption}
\usepackage{enumitem}
\usepackage{amsmath}
\usepackage{tikz}
\usepackage[percent]{overpic}

\newlist{questions}{enumerate}{1}
\setlist[questions]{label=(Q-\arabic*), leftmargin=*, align=left, resume}

\input{ex_shared}
\allowdisplaybreaks[4]
\ifpdf
\hypersetup{
  pdftitle={Neural enrichment finite element method: A hybrid framework for problems with strong oscillations or interface problems},
  pdfauthor={Shihan Guo, Thomas Richter}
}
\fi

\newtheorem{assumption}{Assumption}[section]

\begin{document}

\maketitle

% REQUIRED
\begin{abstract}
We propose a hybrid method, the Neural Enrichment Finite Element Method (NEFEM), designed for problems involving strong oscillations or interface problems with weak discontinuities. This method is based on the stable generalized finite element method (SGFEM) framework, wherein neural networks (NNs) are introduced as enrichment functions for adaptivity, and the Ritz functional is applied for the training process. This works makes two main contributions. First, the method constructs local subspaces with superior approximation properties, significantly reducing the required number of degrees of freedom (DoFs). Second, minimal \emph{a priori} knowledge is required to define enrichment functions, as the NNs evolve heuristically during training. Furthermore, for smooth problems, we provide a residual-based error estimator and prove both its reliability and efficiency. For interface problems, a theoretical analysis on the optimal convergence of the SGFEM is studied, notably without imposing additional regularity assumptions. These analytic results guide the network architecture design and training strategies. The performance and effectiveness of the proposed method is validated through several numerical experiments.  
\end{abstract}

% REQUIRED

% REQUIRED

\section{Introduction}
In this work, we consider solving elliptic problems with strong oscillations as well as interface problems. The standard finite element method (FEM) typically exhibits notable limitations when applied to such problems. To be specific, for problems involving strong oscillations, the number of DoFs in the resulting linear system can be exceedingly large due to high resolution required to achieve meaningful results. When dealing with interface problems, the inherent lack of regularity across the interface causes the overall $H^1$ error for the FEM to deteriorates to $\mathcal{O}(h^{1/2})$
regardless of the polynomial degree $r$ of the finite element space; see the work of Babuška \cite{babuvska1970finite} for further details. 

Many approaches have been proposed for these problems. An important class of methods for solving problems with oscillations, or multiscale characteristics, is the multiscale finite element method (MsFEM) \cite{hou1997multiscale, hou1999convergence,hong2025fem}, where multiscale basis functions are introduced to capture local multiscale information. For interface problems, extensive research has explored both fitted approaches, such as the locally modified parametric FEM by Frei and Richter \cite{frei2014locally}, and unfitted techniques, including the generalized FEM (GFEM) \cite{fries2008corrected, babuvska2017strongly}, the cut FEM \cite{hansbo2014cut}, and the unfitted Nitsche method by Hansbo and Hansbo \cite{hansbo2002unfitted}. A shared characteristic of these approaches is that they all locally modify the underlying finite element space.

For the FEM applied to elliptic problems, Céa's lemma states that
\[
\Vert u-u_h\Vert_V \le C\inf_{w_h\in V_h}\Vert u - w_h\Vert_V,\text{ if } V_h\subset V.
\]
This implies that the subspace solution $u_h$ is the best approximation to the full-space $V$ in respect to the energy norm. Here, the approximation property is heavily dependent on the chosen subspace. In general, we choose piecewise polynomial spaces as the approximation subspace in the FEM, since the approximation property of piecewise polynomial basis is \textit{universal}. 
% For instance, the standard nodal interpolation error estimate using $P_1$ elements is given by
% $$
% \Vert u-\mathcal{I}_hu\Vert_1\le Ch\vert u\vert_2,\quad \forall u\in H^2(\Omega),$$
% demonstrating that functions in $H^2(\Omega)$ can be approximated by piecewise polynomials. 
However, the price of this universality is an excessive number of DoFs, particularly for functions exhibiting strong oscillations or lacking sufficient regularity. This naturally leads to the question:
\begin{questions}
    \item Different subspaces yield varying levels of accuracy. Can we tailor the subspace to each specific problem in order to reduce the number of DoFs?\label{Q1}
\end{questions}
The class of generalized/extended finite element methods (GFEM/XFEM) \cite{strouboulis2001generalized} shed light on how to address \ref{Q1}. Within the GFEM/XFEM framework, additional enrichment functions are introduced to supplement the standard polynomial subspace via the partition of unity method (PUM) \cite{melenk1996partition}. These enrichment functions are designed to capture local characteristics of the solution, enabling the resolution of various types of problems with non-smooth solutions \cite{cui2022stable} and interface problems \cite{wang2025general, gong2024improved}. Despite these advancements, an open question is still ahead of us:
\begin{questions}
    \item Defining enrichment functions typically necessitates \emph{a priori} knowledge, a requirement that is often difficult to satisfy in certain practical scenarios. How can we design adaptive enrichment functions that minimize the reliance on \emph{a priori} knowledge?\label{Q2}
\end{questions}
The universal approximation property of neural networks (NNs) provides a powerful mechanism to answer \ref{Q2}. In fact, deep learning-based numerical solvers for partial differential equations (PDEs) have gained significant attention in recent years. Specifically, several prominent methods have been developed for solving PDEs, such as the Deep Ritz method \cite{yu2018deep}, physics-informed neural networks (PINNs) \cite{raissi2019physics}, and weak adversarial networks \cite{zang2020weak}, which achieve this by designing different types of loss functions, network architectures, and training strategies. One line of research is the hybridization of finite elements and neural networks. The deep neural network multigrid solvers uses neural networks to represent fine scale contributions that cannot be resolved on coarse meshes~\cite{HartmannLessigMargenbergRichter2020,MargenbergJendersieLessigRichter2023}. Here, neural networks are used to predict finite element coefficients on finer meshes, enabling the embedding of the hybrid solution into the finite element space to simplify the analysis \cite{Kapustsin2023}. However, this approach prevents the networks from realizing their full approximation potential.
Several recent studies have explored the integration of the PUM with NNs, offering a promising avenue to answer \ref{Q2}. Baek et al. \cite{baek2024n} proposed a neural network-enriched Partition of Unity (NN-PU) method to adaptively solve boundary value problems via energy minimization, capturing complex local solution features efficiently. Wang et al. \cite{wang2025neural} introduced NNs into the GFEM framework, formulating the Neural Network Element Method. In this approach, local NNs are multiplied by envelope functions, effectively ``pasting'' these local approximations together to yield highly accurate global solutions.

While highly effective in capturing complex local features, both approaches are hindered by inherent drawbacks. In PUM-based methods, accurate numerical integration is notoriously cumbersome due to the irregular supports of the kernel functions. Meanwhile, though GFEM provides accurate results, it often suffers from bad conditioning \cite{babuvska2012stable}. Furthermore, the GFEM suffers from suboptimal convergence rates when addressing interface problems, primarily due to so-called blending effects \cite{fries2010extended}. The stable GFEM (SGFEM), introduced in \cite{babuvska2012stable}, is a modified variant to improve the conditioning. Thus, in addition to \ref{Q1} and \ref{Q2}, this work is also driven by several questions below:
\begin{questions}
    \item How can we incorporate NNs into the SGFEM framework? What are the appropriate loss functions, network architectures, and training strategies for this approach? \label{Q3}
    \item After answering \ref{Q3}, can we devise a reliable and efficient error estimator to dynamically monitor and adaptively guide the training process?\label{Q4}
    \end{questions}
Moreover, recent studies \cite{wang2025general, gong2024improved, zhang2019strongly} demonstrate, both numerically and theoretically, that SGFEM can recover an $\mathcal{O}(h)$ convergence rate in the energy norm for interface problems, provided that so-called general interface conditions are satisfied. However, obtaining this result requires additional regularity assumptions in the analysis.
\begin{questions}
    \item How do the general interface conditions guide the network architecture design for interface problems?\label{Q5}
    \item More importantly, is it possible to obviate this assumption and only require the solution to satisfy $u\in H^1_0(\Omega)\cap H^2(\Omega_0\cup\Omega_1)$?\label{Q6}
\end{questions}

In the subsequent sections, we will provide affirmative answers to all of the questions raised above. In Section \ref{sec:main}, we propose the Neural Enrichment Finite Element Method (NEFEM), thereby addressing \ref{Q2} and \ref{Q3}. Within the same section, we also formulate the estimators for adaptivity, and delve into the network architecture design tailored for interface problems, providing solutions to \ref{Q4} and \ref{Q5}. Section \ref{sec:analysis} establishes the theoretical foundation by proving the reliability and efficiency of the proposed estimator, completing the answer to Question \ref{Q3}. To answer \ref{Q6}, an error analysis of the SGFEM applied to interface problems is presented in the same section. Additionally, several numerical experiments compared with standard FEM are conducted in Section \ref{sec:results}, comprehensively resolving \ref{Q1}. Finally, Section \ref{sec:conclusions} concludes this paper and outlines prospective directions for future research.

\section{Neural enrichment finite element method}
\label{sec:main}
\subsection{Model problems}
  Let $\Omega\subset\mathbb R^d$ be a bounded domain with a piecewise smooth boundary $\partial\Omega$. Denote by $W^{m, q}(\Omega)$ the usual Sobolev spaces defined in $\Omega$ with norm $\Vert\cdot\Vert_{W^{m, q}(\Omega)}$ and semi-norm $\vert\cdot\vert_{W^{m, q}(\Omega)}$. This notation will be replaced by $H^m(\Omega)$ when $q=2$, and $L^q(\Omega)$ when $m=0$. 
  
  We first introduce the problem with strong oscillations. Consider the following elliptic problem:
\begin{equation}
    -\nabla\cdot a(x)\nabla u =f \text{ in }\Omega,\quad u|_{\partial\Omega} = 0, 
    \label{model1}
\end{equation}
where $a\in W^{1,\infty}(\Omega)$ is assumed to be symmetric and positive definite with uniform upper and lower bounds; $f\in L^2(\Omega)$ is the source term. In practice, $a(x)$ and $f$ can be highly oscillatory; therefore, the solution of \eqref{model1} displays strong oscillations and a multi-scale structure. The other type of problems we consider is the interface problems. Let $\Gamma$ be an interface dividing $\Omega$ into two subdomains $\Omega_0$ and $\Omega_1$. Consider the following Laplace equation:
\begin{equation}
    -\nabla\cdot a_i\nabla u =f \text{ in }\Omega,\quad [u]=0,\quad [a\partial_{\boldsymbol{n}} u]=0 \text{ on }\Gamma,
    \label{model2}
\end{equation}
where $a_i>0$ are diffusion parameters.
The notation $[u]=u_0-u_1$ represents the jump on the interface $\Gamma$, where $u_r:=u|_{\bar\Omega_r}, r=0, 1$. $\boldsymbol{n}$ represents the outward pointing unit normal to $\Omega_0$ and $\partial_{\boldsymbol{n}}u=\boldsymbol{n}\cdot\nabla u$. Both subdomains $\Omega_r$ are assumed to have a boundary with sufficient regularity such that for smooth right-hand sides it holds for the solution of \eqref{model2} that
$u\in H^1_0(\Omega)\cap H^2(\Omega_0\cup\Omega_1)$.

The variational formulation of these two kinds of problems is to seek $u\in V=H^1_0(\Omega)$ such that
\begin{equation}
    a(u, v)=f(v), \quad \forall v\in V,
    \label{variation}
\end{equation}
where $a(u, v)=\int_\Omega a\nabla u\cdot\nabla v\,\text{d}x$, and $f(v)=\int_\Omega fv\,\text{d}x$. The energy norm is introduced as $\Vert u\Vert_E=\sqrt{a(u, u)}$, which is equivalent to $\Vert u\Vert_{H^1(\Omega)}$ on $V$. 
Moreover, the solution $u$ can be equivalently characterized as the unique minimizer of the Ritz functional
\begin{equation}
    J(u)\le J(v) := \frac{1}{2}a(v, v) - f(v) \quad \forall v\in V.
    \label{ritz}
\end{equation}
The corresponding discrete problem is to find $u_h\in V_h\subset V$ such that
\begin{equation}
    a(u_h, v_h)=f(v_h), \quad\forall v_h\in V_h.
\end{equation}

\subsection{Generalized finite element method}
Let $\mathcal T_h=\{K\}$ be an unfitted, quasi-uniform, triangle finite element mesh with mesh size $0<h<1$. For interface problems, the interface $\Gamma$ is allowed to traverse the elements, rather than being aligned with element boundaries. $I_h$ and $\mathcal E$ are the set of nodes and internal edges in $\mathcal{T}_h$, respectively. In the GFEM framework, additional basis functions are introduced to enrich the standard polynomial subspace based on the PUM. The standard linear FE hat functions $L_i, i\in I_h$, naturally form a Lipschitz partition of unity:
\begin{equation}
    \Vert L_i\Vert_{L^\infty(\Omega)}=1, \quad\Vert\nabla L_i\Vert_{L^\infty(\Omega)}\le Ch^{-1},\quad\sum_{i\in I_h}L_i\equiv1\text{ in }\Omega.
    \label{partition}
\end{equation}
The compact support of the hat function $L_i$ is denoted as patch $\omega_i$. We denote the set of enriched nodes by $I_h^{enr}\subset I_h$, and non-enriched nodes by $I_h^{std}=I_h\setminus I_h^{enr}$. The GFEM solution is formulated as follows:
\begin{equation*}
u_h = \sum_{i \in \mathcal{I}_h^{std}} u_i L_i 
+ \sum_{i \in \mathcal{I}_h^{enr}} L_i (u_i + a_i \phi_i),
\end{equation*}
   where $\phi_i$ is the enrichment function; $u_i$ and $a_i$ represent the DoF of FE bases and enrichment bases, respectively. 
   In other words, the approximation space consists the standard FE space enriched by the enrichment function:
   \begin{equation*}
       V_h = V_{FEM}\oplus V_{ENR},\qquad V_{ENR}=\operatorname{span}\{L_i\phi_i: i\in I_h^{enr}\}.
   \end{equation*}
    In this work, we employ the SGFEM, which is a stabilized variant of the GFEM designed to improve the conditioning of the resulting system and to eliminate blending effects. 
    Within the SGFEM framework, the approximation space is modified as
    \begin{equation}
    V_h = V_{\mathrm{FEM}} \oplus V_{\mathrm{ENR}},
    \qquad
    V_{\mathrm{ENR}} = \operatorname{span}\{ L_i(\phi_i - \mathcal{I}_h \phi_i) : i \in I_h^{enr} \},
    \label{subspace}
    \end{equation}
    where $\mathcal{I}_h$ denotes the standard linear interpolation operator.
\subsection{Neural enrichment functions}
   To bypass the reliance on a priori knowledge in constructing $\phi_i$, we introduce NNs $\{\mathcal{N}_i\}_{i\in I_h^{enr}}$ as adaptive enrichment functions that heuristically resolve local characteristics. For each enriched node $i \in I_h^{\mathrm{enr}}$, we choose the enrichment function in the form
   \begin{equation}
   \label{neural enrichment}
        \phi_i(x) = \mathcal{N}_i(x - p_i; \theta_i),
    \end{equation}
    where $\theta_i$ denotes trainable parameters, and $p_i$ is the coordinate of the enriched node $i$. Coordinate shifting facilitates optimization by normalizing the network inputs, and improves numerical stability by enhancing the linear independence of the enrichment functions. Let $\theta=\{\theta_i\}, i\in I_h^{enr}$ be the set of trainable parameters. The corresponding subspace can be presented as
   \begin{equation}
       V_{N}(\mathcal{T}_h, I_h^{enr};\theta):=V_{FEM}\oplus\operatorname{span}\{L_i\left[\mathcal{N}_i(\theta_i)-\mathcal{I}_h\mathcal{N}_i(\theta_i)\right]:i\in I_h^{enr}\}.
   \end{equation}
   Let $N$ and $M$ denote the numbers of DoF associated with $V_{FEM}$ and $V_{ENR}$, respectively. 
   % We introduce an ordered basis $\{\varphi_i\}_{i=1}^{N+M}$ of
   %  $V_{N}(\mathcal{T}_h, I_h^{enr};\theta)$ such that
   %  \begin{equation}
   %  V_{N}(\mathcal{T}_h, I_h^{\mathrm{enr}};\theta)
   %  =
   %  \operatorname{span}\left\{
   %  \varphi_1,\dots,\varphi_N,\,
   %  \varphi_{N+1}(\theta_1),\dots,\varphi_{N+M}(\theta_M)
   %  \right\}.
   %  \label{VNe-span}
   %  \end{equation}
    The stiffness matrix $A(\theta)$ and the load vector $F(\theta)$ are defined by
    \begin{equation*}
    A_{ij} = a(\varphi_j, \varphi_i), \quad F_{i} = f(\varphi_i)\qquad 1 \le i,j \le N+M,
    \end{equation*}
    where $\{\varphi_i\}_{i=1}^{N+M}$ is the ordered basis of $V_{N}(\mathcal{T}_h, I_h^{enr};\theta)$.
   The discrete solution $c(\theta)$ is obtained by solving the linear system $A(\theta)c(\theta)=F(\theta)$. To update trainable parameters $\theta$, replacing the solution into the Ritz functional \eqref{ritz}, we have the loss function
   \begin{equation}
           \mathcal{L}(u_h;\theta) = \frac{1}{2}c^T(\theta)A(\theta)c(\theta) - c^T(\theta)F(\theta).
   \end{equation}
   Minimizing the loss function requires the gradient $\frac{\partial\mathcal{L}}{\partial\theta}$.
   Following \cite{aballay2025r}, we can bypass this, as the chain rule reveals
      % tracking the gradient of $\theta$ with respect to $c$ is computationally expensive and challenging, particularly when employing third-party solvers that lie outside of the deep learning framework's automatic differentiation capabilities. Here, we contend that 
   \begin{equation*}
           \frac{\partial \mathcal{L}}{\partial\theta}=\frac{\partial \mathcal{L}}{\partial c}\frac{\partial c}{\partial \theta} + \frac{\partial \mathcal{L}}{\partial A}\frac{\partial A}{\partial\theta} + \frac{\partial \mathcal{L}}{\partial F}\frac{\partial F}{\partial \theta},
   \end{equation*}
   where for the first term, it holds
   \begin{equation*}
       \frac{\partial \mathcal{L}}{\partial c} = A(\theta)c(\theta) - F(\theta) = 0\quad\Rightarrow
    \quad\frac{\partial\mathcal{L}}{\partial\theta}=\frac{\partial \mathcal{L}}{\partial A}\frac{\partial A}{\partial\theta} + \frac{\partial \mathcal{L}}{\partial F}\frac{\partial F}{\partial \theta},
   \end{equation*}
   since $c(\theta)$ is the solution of the linear system. Thus, $\frac{\partial c}{\partial \theta}$ is no longer needed, implying that 
   $c$ can be treated as a constant vector independent of $\theta$. In other word, the framework maintains compatibility with various third-party solvers and legacy codes that are not integrated into deep learning libraries. This allows for the flexible use of both iterative and direct methods, with or without preconditioning. The NEFEM is outlined in Algorithm~\ref{algo:NeFEM} as follows:
\begin{algorithm}[H]
    \caption{Neural enriched finite element method}\label{algo:NeFEM}
    \nl \textbf{Initialization:} Generate the mesh $\{\mathcal{T}_h\}$ and determine the enriched points set $I_h^{enr}$. Build the neural enrichment functions $\phi_i, i\in I_h^{enr}$ and set up $V_{N}(\mathcal{T}_h, I_h^{enr};\theta)$. Choose the maximum training epochs $T$, learning rate $\gamma$ and $k=0$.
    
    \nl \While{$k\le T$}{
            \nl Assemble the stiffness matrix $A$ and the load vector $F$ on $V_{N}(\mathcal{T}_h, I_h^{enr};\theta)$.

            \nl Solve the linear system $Ac=F$ for $c$ without gradient tracking. External solvers are acceptable.

            \nl Compute the loss function $\mathcal{L}$, and update network parameters $\theta$ as follows
            $$
            \theta = \theta - \gamma\frac{\partial\mathcal{L}}{\partial\theta}=\theta-\gamma\left(\frac{\partial \mathcal{L}}{\partial A}\frac{\partial A}{\partial\theta} + \frac{\partial \mathcal{L}}{\partial F}\frac{\partial F}{\partial \theta}\right).
            $$

            \nl $k = k + 1$.
    }
    \nl Assemble the stiffness matrix $A$ and the load vector $F$ on $V_{N}(\mathcal{T}_h, I_h^{enr};\theta)$.
    
    \nl Solve the linear system $Ac=F$ for $c$.
    
    \Return $u_h$
\end{algorithm}

\subsection{Adaptivity for elliptic problems}
In practice, elliptic problems often exhibit only localized oscillations, making it unnecessary to enrich the entire computational domain. For this types of problem, we use the following residual-based estimator: 
\begin{equation}
    \eta^2_K = h_K^2\Vert f+\nabla\cdot a\nabla u_h\Vert^2_{L^2(K)} + \frac{1}{2}h_l\Vert[a\partial_{\boldsymbol{n}}u_h]\Vert^2_{L^2(l)},
\label{estimator}
\end{equation}
to select nodes that need to be enriched. Here, $\eta_K$ denotes the local residual estimator, and $l=\partial K\cap\mathcal{E}$ denotes the interior edges of element $K$. We emphasize that this estimator is also reliable and locally efficient for NEFEM (see Section~\ref{sec:analysis}). This estimator is employed as an error indicator to guide the training process. Specifically, if the estimator indicates sufficient convergence within a local region, the corresponding networks are excluded from further training, thereby accelerating the training process.
\begin{algorithm}[H]
\setcounter{AlgoLine}{0}
    \caption{Adaptive neural enriched finite element method}\label{adaptive}
    \nl \textbf{Initialization:} Generate the mesh $\{\mathcal{T}_h\}$. Set the marking parameters $\alpha_1, \alpha_2\in(0, 1)$, maximum training epochs $T$, adaptive epochs $H_1$ and $H_2$, learning rate $\gamma$ and $k=0$.
    
    \nl Find $u_h\in V_{\mathbb P_1}(\mathcal{T}_h)$ such that $a(u_h, v_h)=f(v_h)\quad\forall v_h\in V_{\mathbb P_1}(\mathcal{T}_h)$.
    
    \nl Compute the local residual estimator $\eta^2_K$ and global error estimator $\eta^2=\sum_K \eta_K^2$.
    
    \nl Construct the marked subset $\mathcal{S}\subset\mathcal{T}_h$ by the percentage marking strategy. Nodes associated with $\mathcal{S}$ are enriched and denoted as $I_h^{enr}$. Build the neural enrichment functions $\phi_i, i\in I_h^{enr}$ and set the subspace $V_{N}(\mathcal{T}_h, I_h^{enr};\theta)$.
    
    \nl \While{$k\le T$}{
            \nl Assemble the stiffness matrix $A$ and the load vector $F$ on $V_{N}(\mathcal{T}_h, I_h^{enr};\theta)$.

            \nl Solve the linear system $Ac=F$ for $c$ without gradient tracking. External solvers are acceptable.

            \nl \If{$k\ge H_1$ and $(k - H_1) \equiv 0 \pmod{H_2}$}{
                \nl Compute the local residual estimator $\eta^2_K$ and global error estimator $\eta^2$.
    
            \nl Dörfler Marking Strategy: Find the training elements $\mathcal{S}_t\subset\mathcal{S}$ such that \eqref{Dorfler}. Nodes associated with $\mathcal{S}_t$ are denoted as $I_t$.
            }

            \nl Compute the loss function $\mathcal{L}$, and update network parameters $\theta_{I_t}$ as follows
            $$
            \theta_{I_t} = \theta_{I_t} - \gamma\frac{\partial\mathcal{L}}{\partial\theta_{I_t}}=\theta_{I_t}-\gamma\left(\frac{\partial \mathcal{L}}{\partial A}\frac{\partial A}{\partial\theta_{I_t}} + \frac{\partial \mathcal{L}}{\partial F}\frac{\partial F}{\partial \theta_{I_t}}\right).
            $$

            \nl $k = k + 1$.
    }
    \nl Assemble the stiffness matrix $A$ and the load vector $F$ on $V_{N}(\mathcal{T}_h, I_h^{enr};\theta)$.
    
    \nl Solve the linear system $Ac=F$ for $c$.
    
    \Return $u_h$
\end{algorithm}

Regarding the marking strategy, we employ a percentage-based marking for node selection and the Dörfler marking for the training process. Specifically, we construct the marked subset $\mathcal{S} \subset \mathcal{T}_h$ by marking the top $\alpha_1\in(0, 1)$ of elements with the largest error indicators. We then define $I^s \subset I_h^{\mathrm{enr}}$ as the set of enriched nodes associated with the elements in $\mathcal{S}$; the enrichment functions corresponding to these nodes are then actively trained. For the Dörfler marking strategy used in the training process, we seek a subset $\mathcal{S}_t \subset \mathcal{T}_h$ of minimal cardinality that satisfies
\begin{equation}
    \sum_{K \in \mathcal{S}_t} \eta_K^2 \ge \alpha_2 \sum_{K\in\mathcal{T}_h}\eta_K^2,
    \label{Dorfler}
\end{equation}
where $\alpha_2 \in (0,1)$ is a prescribed parameter. 
% In practice, we achieve this by sorting the elements in descending order of $\eta_K$ and greedily collecting those with the highest errors into $\mathcal{S}_t$ until the condition (\ref{Dorfler}) is reached. 
The adaptive NEFEM is outlined in Algorithm~\ref{adaptive}, which is reminiscent of $p-$adaptivity in FEM.
\subsection{Network architectures for interface problems}
Interface problems feature a flux discontinuity across the interface. Using smooth activation function, such as Tanh or Sigmoid, to express this flux discontinuity is generally difficult.  While non-smooth activation functions, for example ReLU, naturally contain this discontinuity, they introduce a more severe issue: the position of the discontinuity may shift during training process. This mobility adversely effects the integration accuracy, as sub-element quadrature rules are typically pre-aligned with the physical interface.

To address this issue, we retain the use of smooth activation functions but augment the input space with an additional interface-aware variable. Specifically, we define the enrichment function as
\begin{equation}
    \phi_i({x}, D({x})) = \mathcal{N}_i({x}-{p}_i, D({x});\theta_i),
    \label{interface_neural}
\end{equation}
where $D(x)$ is a quasi-distance function associated with the interface $\Gamma$ and satisfies
\begin{equation}
        (1)\; D\in W^{2, \infty}(\bar\Omega_r), r=0, 1 \qquad 
        (2)\; [D] = 0\qquad 
        (3)\; 0<\kappa_1\le\left\vert[\partial_nD]\right\vert\le\kappa_2.
        \label{quasi}
\end{equation}
Conditions \eqref{quasi} are also called general enrichment conditions for interface problems \cite{wang2025general}. 
% Although the neural network $\mathcal{N}_i$ employs smooth activation functions, the composition with $D(\boldsymbol{x})$ introduces a discontinuity in the normal derivative of $\phi_i$ across the interface. In fact, 
By chain rule we have
\begin{equation*}
    \nabla \phi_i = \partial_x \mathcal{N}_i + \partial_D\mathcal{N}_i \, \nabla D .
\end{equation*}
From the third condition in \eqref{quasi}, the second term induces a flux discontinuity for $\phi_i$. Importantly, this kink is anchored through $D(x)$ and therefore does not shift during the training process. The network structure is illustrated in Figure \ref{structure}.
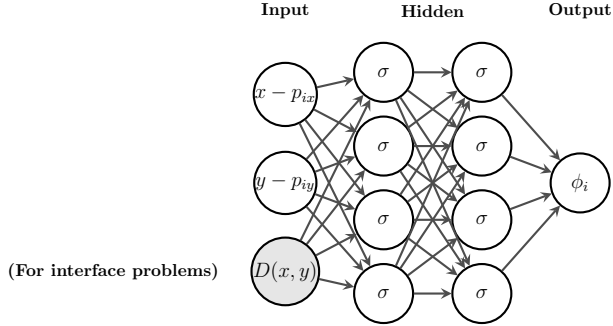
\begin{figure}[htbp]
\centering
\begin{tikzpicture}[
    scale=0.65, transform shape,
    neuron/.style={circle, draw, thick, minimum size=1.2cm, inner sep=0pt, font=\large},
    link/.style={->, >=stealth, thick, draw=black!70}
]

    \node[neuron] (I-1) at (0, 1.8) {$x-p_{ix}$};
    \node[neuron] (I-2) at (0, 0) {$y-p_{iy}$};
    \node[neuron, fill=black!10] (I-3) at (0, -1.8) {$D(x, y)$};

    \node[neuron] (H1-1) at (2, 2.25) {$\sigma$};
    \node[neuron] (H1-2) at (2, 0.75) {$\sigma$};
    \node[neuron] (H1-3) at (2, -0.75) {$\sigma$};
    \node[neuron] (H1-4) at (2, -2.25) {$\sigma$};

    \node[neuron] (H2-1) at (4, 2.25) {$\sigma$};
    \node[neuron] (H2-2) at (4, 0.75) {$\sigma$};
    \node[neuron] (H2-3) at (4, -0.75) {$\sigma$};
    \node[neuron] (H2-4) at (4, -2.25) {$\sigma$};

    \node[neuron] (O-1) at (6, 0) {$\phi_i$};

    \foreach \i in {1, 2, 3} {
        \foreach \j in {1, 2, 3, 4} {
            \draw[link] (I-\i) -- (H1-\j);
        }
    }
    
    \foreach \i in {1, 2, 3, 4} {
        \foreach \j in {1, 2, 3, 4} {
            \draw[link] (H1-\i) -- (H2-\j);
        }
    }
    
    \foreach \j in {1, 2, 3, 4} {
        \draw[link] (H2-\j) -- (O-1);
    }

    \node at (0, 3.5) {\textbf{Input}};
    \node at (3, 3.5) {\textbf{Hidden}};
    \node at (6, 3.5) {\textbf{Output}};

    \node at (-3.5, -1.8) {\textbf{(For interface problems)}};
\end{tikzpicture}
\caption{Illustration of the network structure for each neural enrichment function $\phi_i$.}
\label{structure}
\end{figure}
\begin{remark}
    Condition~\eqref{quasi} is satisfied by a wide range of commonly used interface representations. Typical examples include the distance function $\operatorname{dist}(x, \Gamma)$, and the absolute value of a level-set function $|\varphi|$.
\end{remark}
% Conditions \eqref{quasi} are also called general enrichment conditions for interface problems \cite{}. For enrichment functions satisfying \eqref{quasi}, SGFEM achieves an $\mathcal{O}(h)$ convergence rate in the energy norm for interface problems. In the next section, we investigate this condition under weaker regularity assumptions than those commonly imposed in the existing literature\cite{}.

\section{Numerical analysis}
\label{sec:analysis}
% In this section, we present the rigorous theoretical analysis of NEFEM. First, the reliability and efficiency of the estimator \eqref{estimator} are investigated. Furthermore, we establish error estimates for the approximations employing enrichment functions that satisfy~\eqref{quasi} for interface problems, notably achieving this under weaker regularity assumptions compared to existing works.
\subsection{Estimator}
\label{sec:error estimates}
Algorithm~\ref{adaptive} utilizes the expression in \eqref{estimator} as an \textit{a posteriori} error indicator to evaluate the NEFEM approximation. Although this estimator is well-documented in standard FEM literature \cite{chamoin2023introductory}, its theoretical validity for neural enrichment requires formal justification. To this end, the subsequent theorem proves that the estimator remains reliable when applied to NEFEM.
\begin{theorem}[Reliability]
    Let $u$ be the solution of \eqref{model1}; $u_h\in V_{N}(\mathcal{T}_h, I_h^{enr}; \theta)$ be the approximation solution. Let $\eta_K$ be the estimator defined in \eqref{estimator}. Then, there exists a positive constant $C$ such that
    \begin{equation}\label{reliability}
        \Vert u-u_h\Vert^2_E\le C\sum_{K\in\mathcal{T}_h}\eta_K^2.
    \end{equation}
\end{theorem}
\begin{proof}
    Let $e=u-u_h$ be the approximation error. Denote by $R(v)=f(v)-a(u_h, v)$ the residual functional defined on $H^{-1}(\Omega)$. From direct derivations we have the following relation
    \begin{equation}
        a(e, v) = a(u-u_h, v) = f(v) - a(u_h, v) = R(v) \quad \forall v\in V.
        \label{star}
    \end{equation}
    Together with the Galerkin orthogonality,
    \begin{equation}
        \Vert e\Vert_E^2=a(e, e) = a(e, e-v_h)=R(e-v_h), \quad \forall v_h\in V_{Ne}(\mathcal{T}_h, I_h^{enr}; \theta).
        \label{by part}
    \end{equation}
    On the other hand, applying integration by parts yields
    \begin{equation}
    \label{weak form}
        R(v) = \sum_{K\in\mathcal{T}_h}(f + \nabla\cdot a\nabla u_h, v)+\sum_{l\in\mathcal{E}}\frac{1}{2}\langle[a\partial_{\boldsymbol{n}}u_h],v\rangle.
    \end{equation}
    Here, $\langle \cdot, \cdot\rangle$ denotes the $L^2$ inner product on the trace domain. Note that $\mathbb{P}_1\subset V_{N}(\mathcal{T}_h, I_h^{enr}; \theta)$. Set $v_h=\Pi_h e$, where $\Pi_h$ is the Clément interpolate operator. Combining the approximation results of the Clément interpolation, the Cauchy-Schwartz inequality, \eqref{by part} and \eqref{weak form}, the following estimates hold:
    \begin{equation*}
    \begin{aligned}
        &\Vert e\Vert_E^2 =\sum_{K\in\mathcal{T}_h}(f + \nabla\cdot a\nabla u_h, e-\Pi_he)+\sum_{l\in\mathcal{E}}\frac{1}{2}\langle[a\partial_{\boldsymbol{n}}u_h],e-\Pi_he\rangle\\
        &\le\sum_{K\in\mathcal{T}_h}Ch_K\Vert\nabla e\Vert_{L^2(K)}\Vert f + \nabla\cdot a\nabla u_h\Vert_{L^2(K)}+\frac{1}{2}\sum_{l\in\mathcal{E}}Ch_l^{1/2}\Vert\nabla e\Vert_{L^2(l)}\Vert[a\partial_{\boldsymbol{n}}u_h]\Vert_{L^2(l)}\\
        &\le C\left(\sum_{K\in\mathcal{T}_h} h_k^2\Vert f+\nabla\cdot a\nabla u_h\Vert^2_{L^2(K)} + \sum_{l\in\mathcal{E}}\frac{1}{2}h_l\Vert[a\partial_{\boldsymbol{n}}u_h]\Vert^2_{L^2(l)}\right)^{1/2}\Vert e\Vert_E.
        \end{aligned}
    \end{equation*}
    Then, \eqref{reliability} holds from direct computation.
\end{proof}

Let $\phi_j$, $j=1,\dots,M,$ be enrichment functions associated with an element $K\subset\bar{\Omega}$. We define the local enrichment function as $r^j_{h, K}:=(\phi_j-\mathcal{I}_h\phi_j)|_K$, and the corresponding local space as 
$$V_N(K):=\mathbb{P}_1(K) \oplus \operatorname{span}\{L_jr^j_{h, K}\}_{j=1}^{M}.$$ 
By mapping $\hat{V}_{h, K}$ from physical elements to the reference element $\hat K$, we obtain the reference space $V_N(\hat{K}):=\mathbb P_1(\hat{K}) \oplus \operatorname{span}\{\hat{L}_j\hat{r}^j_{h, K}\}_{j=1}^M$. To establish the local effectivity of the estimator, we impose the following assumption.
\begin{assumption}
\label{assumption}
   For all mesh size $h$ and element $K$, let the basis $\{\hat{L}_j\hat{r}^j_{h, K}\}$ be uniformly linearly independent. That is, there exists a constant $C_G>0$, independent of $h$ and $K$, such that
   $$
   \alpha^TG_{h, K}\alpha\ge C_G|\alpha|^2,\quad\forall a\in\mathbb{R}^M,
   $$
   where
   $$G_{h, K}:=((\hat{L}_i\hat{r}^i_{h, K}, \hat{L}_j\hat{r}^j_{h, K})_{L^2(\hat{K})})^M_{i, j=1}$$
   is the $L^2-$Gram matrix.
\end{assumption}
\begin{remark}
        From a numerical perspective, if two enrichment functions exhibit strong linear dependence, one of them becomes redundant and is subsequently removed from the set. Moreover, the definition of enrichment functions in \eqref{neural enrichment} alleviates the issue of linear dependence to some extent.
\end{remark}
\begin{lemma}[Inverse inequality]
\label{Inverse inequality}
Suppose that all enrichment functions $\phi_j$ associated with an element $K\subset\bar{\Omega}$ satisfy Assumption~\ref{assumption}. Then, there exists a constant $C>0$ independent of $h$ and $K$ such that 
\begin{equation}
    \Vert \nabla v_h\Vert_{L^2(K)}\le Ch^{-1}\Vert v_h\Vert_{L^2(K)}
    \label{inverse}
\end{equation}
for all $v_h\in V_N(K)$.
\end{lemma}
\begin{proof}
    We employ a scaling argument by considering the function space $\hat{V}_{h, K}$ defined on the reference element. The subsequent proof is structured into three steps.

    (Step 1) First, we show the inverse estimate in $R_{h, K}:=\operatorname{span}\{\hat{L}_j\hat{r}^j_{h, K}\}_{j=1}^M$. For any $w\in R_{h, K}$, we can write $w=\sum_{i=1}^Ma_i\hat{L}_i\hat{r}^i_{h, K}.$ It follows from Assumption \ref{assumption} that 
    $$
    \Vert w\Vert^2_{L^2(\hat K)}=\sum^M_{i, j=1}\alpha_i\alpha_j(\hat{L}_i\hat{r}^i_{h, K}, \hat{L}_j\hat{r}^j_{h, K})_{L^2(\hat{K})}=\alpha^TG_{h,K}\alpha\ge C_G|\alpha|^2.
    $$
    This implies 
    \begin{equation}
        |\alpha|\le C_G^{-1/2}\Vert w\Vert_{L^2(\hat K)}.
        \label{|a|}
    \end{equation}
    On the other hand, from the interpolation error we have
    $$
    \Vert\hat{r}^j_{h, K}\Vert_{H^2(\hat{K})}\le C\vert\phi_j\vert_{H^2(\hat K)}\le C h^{2-d/2}_K\vert\phi_j\vert_{H^2(\Omega)}.
    $$
    Therefore, $\mathcal{F}^j:=\{\hat{r}^j_{h, K}\}$ is uniformly bounded in $H^2(\hat{K})$. Combining \eqref{|a|} yields
    \begin{equation}
            \Vert w\Vert_{H^1(\hat{K})}\le\sum^M_{i=1}|\alpha_i|\Vert \hat{L}_i\Vert_{H^1(\hat K)}\Vert \hat{r}^i_{h, K}\Vert_{H^1(\hat K)}\le C\sqrt{M}|\alpha|\le C_R\Vert w\Vert_{L^2(\hat{K})}.
            \label{inverse_Rn}
    \end{equation}
    
    (Step 2) Next, we show that $R_{h, K}$ and $\mathbb{P}_1(\hat K)$ are uniformly separate. That is, there exists $\delta>0$ independent of $h$ and $K$, such that
    $$
    \operatorname{dist}_{L^2}(w, \mathbb{P}_1(\hat K))\ge\delta\Vert w\Vert_{L^2(\hat{K})}, \quad\forall w\in R_{h, K}.$$ 
    If not, there exists $w_n\in R_{h_n, K_n}$ such that $\Vert w_n\Vert_{L^2(\hat{K})}=1$ but 
    $\operatorname{dist}_{L^2}(w_n, \mathbb{P}_1(\hat K))\to 0.$
    Since the embedding $H^2(\hat K)\hookrightarrow H^1(\hat K)$ is compact, $\mathcal{F}=\prod_{j=1}^M\mathcal{F}^j$ is therefore relatively compact in $H^1(\hat K)$. We extract a convergent subsequence such that $w_n\to w_*$ in $H^1(\hat{K})$. On the other hand, we extract  $p_n\in\mathbb{P}_1(\hat K)$ such that $\Vert w_n-p_n\Vert_{L^2(\hat K)}\to 0.$ Therefore, $p_n\to w_*$ in $L^2(\hat{K}).$ As $\mathbb{P}_1(\hat K)$ is closed, $w_*\in\mathbb{P}_1(\hat K)$. Note that all functions in $R_{h, K}$ vanish at vertices of $\hat K$. From the Sobolev embedding theorem, $w_*$ vanishes at each vertex of $\hat{K}$ when dimension $d\le3$. We deduce that $w_* \equiv 0$, which contradicts the condition $\Vert w_*\Vert_{L^2(\hat K)}=1$.

    (Step 3) Finally, we show the inverse inequality \eqref{inverse}. For any $h, K$ and $v\in \hat{V}_{h, K}$, we write it as $v = p+w$, where $p\in\mathbb{P}_1(\hat{K})$ and $w\in R_{h, K}$. From step 2,
    \begin{equation}
    \Vert v\Vert_{L^2(\hat{K})}=\Vert p+w\Vert_{L^2(\hat K)}\ge\operatorname{dist}_{L^2}(w, \mathbb{P}_1(\hat{K}))\ge\delta\Vert w\Vert_{L^2(\hat{K})}.
    \label{w_v}
    \end{equation}
    Thus, together with the norm equivalency result in $\mathbb{P}_1(\hat{K})$, we have
    \begin{equation}
        \Vert p\Vert_{H^1(\hat K)}\le C_P(1+\delta^{-1})\Vert v\Vert_{L^2(\hat K)}.
        \label{p_in}
    \end{equation}
    Ultimately, combining \eqref{inverse_Rn}, \eqref{w_v} and \eqref{p_in} we yield
    \begin{equation}
        \Vert \nabla v\Vert_{L^2(\hat K)}\le C\Vert v\Vert_{L^2(\hat{K})}.
    \end{equation}
    By scaling argument, one immediately obtains \eqref{inverse}.
\end{proof}
\begin{remark}
    The purpose of Assumption \ref{assumption} is to ensure that the enrichment space $R_{h, K}$ does not collapse into $\mathbb{P}_1(\hat{K)}$. In the case of single global enrichment function $\phi$, which is standard practice in classical SGFEM literature, the situation simplifies. To maintain the flow of the presentation, we provide a more straightforward assumption, and derive the inverse estimate for this case in the Appendix.
\end{remark}

To establish the efficiency result, we follow Verfürth's argument \cite{verfurth1994posteriori}. We start by introducing a bubble function $b_K$ such that $b_K\in H_0^1(K)$ with $b_K>0$ on $\mathring{K}$. As a consequence of Lemma~\ref{Inverse inequality}, the following inverse estimates hold.
\begin{lemma}
     Let $b_K$ be the bubble function defined on $K$. Suppose that all enrichment functions $\phi_j$ associated with an element $K\subset\bar{\Omega}$ satisfy Assumption~\ref{assumption}. There exists constants $C>0$ independent of $h$ and $K$ such that
    \begin{align}
        \Vert v_h\Vert^2_{L^2(K)}&\le C\int_Kv_h^2b_K\,\text{d}x \quad &&\text{(Norm equivalency)}\label{norm equivalency},\\
        \Vert \nabla(b_Kv_h)\Vert_{L^2(K)}&\le Ch^{-1}\Vert v_h\Vert_{L^2(K)} \quad &&\text{(Inverse inequality)},\label{inverse}
    \end{align}
for all $v_h\in V_N(K)$.
\end{lemma}
\begin{proof}
(1) We employ a scaling argument, and show that there exists a constant $C$ independent of $h$ and $K$ such that
\begin{equation}
\Vert \hat{v}_h\Vert^2_{L^2(\hat K)}\le C\int_{\hat K}\hat{v}_h^2b_{\hat K}\,\text{d}x.
\end{equation}
If not, there exists $v_n\in\hat{V}_{h_n, K_n}$ such that $\Vert v_n\Vert_{L^2(\hat K)}=1$ but $\Vert b^{1/2}_{\hat K}v_n\Vert_{L^2(\hat K)}\to 0$. From the inverse inequality Lemma \ref{Inverse inequality}, $\{v_n\}$ is uniformly bounded in $H^1(\hat{K})$, implying that there exists a subsequence such that $v_{n_k}\to v$ in $L^2(\hat K)$. As $b^{1/2}_{\hat K}\in L^\infty(\hat K)$, $b^{1/2}_{\hat K}v_{n_k}\to b^{1/2}_{\hat K}v$ in $L^2(\hat K)$. Thus, $v=0$  a.e. on $K$, which contradicts the condition $\Vert v\Vert_{L^2(\hat K)}=1$.

(2) From the inverse inequality Lemma \ref{Inverse inequality} and the Hölder inequality, we have
\begin{align*}
    \Vert \nabla(b_Kv_h)\Vert_{L^2(K)}\le\Vert\nabla b_K\Vert_{L^\infty}\Vert v_h\Vert_{L^2(K)} + \Vert b_K\Vert_{L^\infty}\Vert\nabla v_h\Vert_{L^2(K)}
    \le Ch^{-1}\Vert v_h\Vert_{L^2(K)}.
\end{align*}
\end{proof}

Analogous to Lemma \ref{Inverse inequality}, to handle the flux jump terms on element edges, our goal is to establish an inverse estimate of the following form
\begin{equation*}
    \Vert \nabla v_h\Vert_{L^2(l)}\le Ch^{-1}\Vert v_h\Vert_{L^2(l)}, \quad \forall v_h\in W_l,
\end{equation*}
where the definition of $l$ and $W_l$ will be provided shortly. To this end, we introduce a further assumption that is similar to Assumption \ref{assumption} but related to the element edges. Let $l$ be an interior edge shared by the elements $K^+$ and $K^-$, and define $\omega_l := K^+ \cup K^-$. Let $r^j_{h, l}:=(\phi_j - \mathcal{I}_h\phi_j)|_l$ denote the enrichment functions associated with patch $\omega_l$ and restricted to $l$. We first demonstrate the approximation space on $l$. Let $v^+ = p^++\sum a_j^+L_jr^j$ and $v^- = p^-+\sum a_j^-L_jr^j$ be functions defined on $K^+$ and $K^-$, respectively. Here, the FE hat function $L_j|_l=0$ if the node $j$ associated with $L_j$ is not on $e$. Moreover, $\phi^j$ and $\partial_n\phi^j$ are continuous across $l$. Thus,
\begin{equation}
    w_h:= \partial_nv^+|_l - \partial_nv^-|_l = 
    \underbrace{[\partial_n p]}_{\in \mathbb{P}_0(l)} + 
    \underbrace{\sum \alpha_j[\partial_nL_j]r_{h, l}^j}_{\in \operatorname{span}\{r^j_{h, l}\}} + 
    \underbrace{\sum \alpha_jL_j|_l[\partial_n(\phi^j-\mathcal{I}_h\phi^j)]}_{\in \mathbb{P}_1(l)},
\end{equation}
which implies that the approximation space on $l$ is
\begin{equation}
    w_h\in W_{l}:=\mathbb{P}_1(l)\oplus\operatorname{span}\{r_{h, l}^j\}_{j=1}^{M^++M^-}.
    \label{flux space}
\end{equation}
The following assumption ensures that $W_{l}$ is non-degenerated.
\begin{assumption}
\label{edge assumption}
    For all mesh size $h$ and the patch $\omega_l$, the basis functions $\{\hat{r}^j_{h, l}\}$ defined on the reference patch $\hat\omega_l$ and associated with $\omega_l$ are uniformly linearly independent. That is, there exists a constant $C_l>0$, independent of $h$ and $\omega_l$, such that
   $$
   \alpha^TG_{h, l}\alpha\ge C_l|\alpha|^2,\quad\forall \alpha\in\mathbb{R}^{M^++M^-},
   $$
   where
   $$G_{h, l}:=((\hat{r}^i_{h, l}, \hat{r}^j_{h, l})_{L^2(\hat{l})})^{M^++M^-}_{i, j=1}$$
   is the $L^2-$Gram matrix.
\end{assumption}
Based on Assumption \ref{edge assumption}, we have the following inverse estimates on $l$.
\begin{lemma}[Inverse inequality]
\label{edge inequality}
Suppose that all enrichment functions $\phi_j$ associated with a patch $\omega_l$ satisfy Assumption~\ref{edge assumption}. Then, there exists a constant $C>0$ independent of $h$ and $\omega_l$ such that 
\begin{equation}
    \Vert \nabla v_h\Vert_{L^2(l)}\le Ch^{-1}\Vert v_h\Vert_{L^2(l)}
    \label{inverse}
\end{equation}
for all $v_h\in W_l$.
\end{lemma}
\begin{proof}
    The proof follows the same lines as that of Lemma \ref{Inverse inequality}.
\end{proof}

Next, we introduce the edge bubble function $b_l := L_i L_j$, where $L_i$ and $L_j$ are the nodal basis functions associated with the endpoints of $l$. The support of $b_l$ is exactly $\omega_l$. As a consequence of Assumption \ref{edge assumption} and Lemma \ref{edge inequality}, we have the following norm equivalency results.
\begin{lemma}[Norm equivalency]
    Consider the enrichment function $\phi$ satisfy Assumption~\ref{edge assumption}, and let $b_l$ be the bubble function defined on $\omega_l$. Then, there exists a constant $C>0$ independent of $h$ and $\omega_l$ such that
    \begin{equation}
        \Vert w_h\Vert_{L^2(l)}^2\le C\int_l w^2_hb_lds
        \label{edge equivalency}
    \end{equation}
    for all $w_h\in W_l$
    \end{lemma}
\begin{proof}
    The proof follows the same lines as that of the norm equivalency \eqref{norm equivalency}.
\end{proof}

We further introduce the normal-constant extension $T_l: H^1(l)\to H^1(\omega_l)$. Specifically, for any point in $\omega_l$, let $(s, n)$ be its local coordinates where $s$ denotes the projection onto the hyperplane containing $l$, and $n$ denotes the coordinate in the normal direction. For a given function $v \in H^1(l)$, the extension is defined such that it remains constant along the normal direction, i.e., $(T_l v)(s, n) = v(s)$. Then, the following inverse estimates hold:
\begin{lemma}[Inverse inequality]
    Consider the enrichment function $\phi$ satisfy Assumption~\ref{edge assumption}, and let $b_l$ be the bubble function defined on $\omega_l$. Then, there exists constants $C>0$ independent of $h$ and $\omega_l$ such that
        \begin{align}
        \Vert b_l T_lw_h\Vert_{L^2(\omega_l)}&\le Ch^{1/2}\Vert w_h\Vert_{L^2(l)},\label{edge_inv1}\\
                \Vert \nabla(b_l T_lw_h)\Vert_{L^2(\omega_l)}&\le Ch^{-1/2}\Vert w_h\Vert_{L^2(l)}.\label{edge_inv2}
    \end{align}
\end{lemma}
\begin{proof}
    We employ a scaling argument, and show that there exists a constant $C$ independent of $h$ and $\omega_l$ such that
    \begin{equation}
            \Vert b_{\hat l} \hat T_{l}\hat {w}_h\Vert_{L^2(\hat{\omega}_l)}\le C\Vert \hat w_h\Vert_{H^1(\hat l)}.
            \label{reference}
    \end{equation}
    To begin with, the normal-constant extension $\hat{T}_l$ is $H^1$ stable, i.e., 
    $
        \Vert \hat{T}_l\hat w\Vert_{H^1({\hat{\omega}_l})}\le C\Vert \hat w\Vert_{H^1(\hat l)}
    $
    for any $\hat w\in H^1(\hat{l})$.
    Together with the inverse inequality \eqref{inverse}, we obtain
    $
         \Vert \hat{T}_l\hat w\Vert_{H^1({\hat{\omega}_l})}\le C\Vert \hat w\Vert_{L^2(\hat l)}.
    $
    As $\hat{b}_l\in W^{1, \infty}(\hat\omega_l)$, \eqref{reference} holds from direct derivations together with the Hölder inequality.
\end{proof}

We only consider here a simple situation, where $a(x)$ is a piecewise polynomial function of degree 1. Let $\omega_K$ be a neighbor of $K$; $f_h$ be a local approximation of $f$ in $V_N(\mathcal{T}_h, I^{enr}_h; \theta)$. The local efficiency result is presented as follows.
\begin{theorem}[Local efficiency]
     There exists a constant $C$, depending on enrichment functions $\{\phi_i\}$ applied in $\omega_K$ and the shape regularity of the mesh, but independent of $h$ and location of $K$, such that
    \begin{equation}
        \eta_K\le C\left(\Vert u-u_h\Vert_{E,\omega_K}+\operatorname{osc}(f, \omega_K)\right),
        \label{local efficiency}
    \end{equation}
    where $\omega_K$ is a neighborhood of $K$, and $\operatorname{osc}(f, \omega_K)$ is the data oscillation defined as
    \begin{equation}
        \operatorname{osc}(f, \omega_K):=\sqrt{\sum_{K\subset\omega_K}h^2_K\Vert f-f_h\Vert^2_{L^2(K)}}.
    \end{equation}
\end{theorem}
\begin{proof}
    Let $R_K=f + \nabla a\cdot\nabla u_h$. Since $R_K = (f - f_h) + f_h + \nabla a\cdot\nabla u_h=:(f - f_h) + R_{K, h}$, we first only need to show
    \begin{equation}
    h\Vert R_{K, h}\Vert_{L^2(K)}\le C\left(\Vert e\Vert_{E, K} + h\Vert f-f_h\Vert_{L^2(K)}\right).
    \label{st1}
    \end{equation}
    Direct computation shows that $R_{K, h}= (f_h - f) - \nabla a\cdot \nabla e$. Let $v_b =b_KR_{K, h}$, where $b_K$ is the bubble function defined previously. Testing the equation with $v_b$, and  applying the norm equivalency \eqref{norm equivalency} and inverse inequality \eqref{inverse} yield
    \begin{align*}
        \Vert R_{K,h}\Vert^2_{L^2(K)}&\le C(R_{K, h}, v_b) = C(f-f_h, v_b) - C(\nabla a\cdot \nabla e,v_b)\\
        &\le C\Vert f - f_h\Vert_{L^2(K)}\Vert b_KR_{K,h}\Vert_{L^2(K)} + C\Vert e\Vert_{E, K}\Vert \nabla(b_KR_{K,h})\Vert_{L^2(K)}\\
        &\le C\Vert f - f_h\Vert_{L^2(K)}\Vert R_{K,h}\Vert_{L^2(K)} + Ch^{-1}\Vert e\Vert_{E, K}\Vert R_{K,h}\Vert_{L^2(K)}.
    \end{align*}
    By multiplying both side by $h$ and dividing by $\Vert R_{K,h}\Vert_{L^2(K)}$, we obtain \eqref{st1}.

    Next, we demonstrate that
    \begin{equation}
        h^{1/2}\Vert [a\nabla_nu_h]\Vert_{L^2(l)}\le C\left(\Vert e\Vert_{E,\omega_l} + \sum_{K\subset\omega_K}h\Vert R_{K}\Vert_{L^2(\omega_K)}\right).
        \label{st2}
    \end{equation}
    Let $v_l=b_lT_l[a\nabla_nu_h]$ be the test function, where $b_l$ and $T_l$ are the bubble function and extension. Testing the equation $-\nabla a\cdot\nabla e = R_K$ with $v_l$, and applying the norm equivalency \eqref{edge equivalency} and the inverse inequalities \eqref{edge_inv1} and \eqref{edge_inv2}, we have
    \begin{align*}
    \Vert [a\nabla_nu_h]\Vert^2_{L^2(l)} &\le C\langle [a\nabla_nu_h], v_l\rangle = C(a\nabla e,\nabla v_l)-C(R_K, v_l)\\
    &\le C\Vert R_K\Vert_{L^2(\omega_l)}\Vert v_l\Vert_{L^2(\omega_l)} + C\Vert e\Vert_{E,\omega_e}\Vert\nabla v_l\Vert_{L^2(\omega_e)}\\
    &\le Ch^{1/2}\Vert R_K\Vert_{L^2(\omega_l)}\Vert [a\nabla_nu_h]\Vert_{L^2(l)} + Ch^{-1/2}\Vert e\Vert_{E,\omega_e}\Vert [a\nabla_nu_h]\Vert^2_{L^2(l)}.
    \end{align*}
    Therefore, by multiplying both side by $h^{1/2}$ and dividing by $\Vert [a\nabla_nu_h]\Vert_{L^2(l)}$, we therefore obtain \eqref{st2}. Combining \eqref{st1} and \eqref{st2}, we yield the local efficiency \eqref{local efficiency}.
\end{proof}

\subsection{Error estimates on interface problems}
Let $I^{enr}_h=\{i\in I_h: P_i\in e_s \text{ where } e_s\cap\Gamma\neq\emptyset\}$ be the set of enriched nodes. Note that the enrichment strategy here is to only enrich those elements intersected by the interface, see Figure \ref{fig:int}. To obtain the optimal convergence rate $\mathcal{O}(h)$, the function space
\begin{equation}
\mathbb M:=\{u: u|_{\Omega_r}\in H^2(\Omega_r),\; r=0, 1 \text{ and } \Vert D^\alpha u\Vert_{L^\infty(\Gamma)}<\infty, \vert\alpha\vert\le1\}
\end{equation}
with norm
\begin{equation}
\Vert u\Vert_{\mathbb M} = \Vert u_0\Vert_{H^2(\Omega_0)}+\Vert u_1\Vert_{H^2(\Omega_1)}+\sum_{\vert\alpha\vert\le1}\Vert D^\alpha u\Vert_{L^\infty(\Gamma)}
\end{equation}
is introduced. This is a prevailing assumption \cite{wang2025general, gong2024improved, zhang2019strongly} in the analytical study of SGFEM for interface problems. 
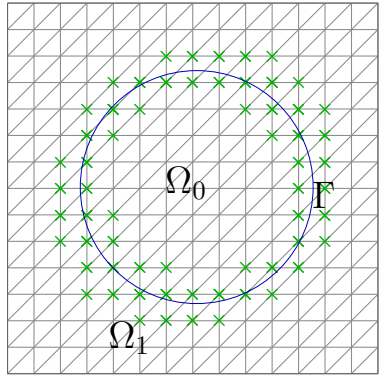
\begin{figure}[htbp]
    \centering
\begin{tikzpicture}[scale=0.35]

% --------------------------------
% parameters
% --------------------------------
\def\N{7}          % grid from -N to N
\def\R{4.4}        % radius of the interface
\def\xc{0.15}      % x-coordinate of circle center
\def\yc{0.05}      % y-coordinate of circle center

% --------------------------------
% draw structured square grid
% --------------------------------
\foreach \i in {-\N,...,\N} {
  \draw[thin, black!45] (\i,-\N) -- (\i,\N);
  \draw[thin, black!45] (-\N,\i) -- (\N,\i);
}

% --------------------------------
% split each square into two right triangles
% using the same diagonal direction
% --------------------------------
\foreach \i in {-7,...,6} {
  \foreach \j in {-7,...,6} {
    \draw[thin, black!45] (\i,\j) -- (\i+1,\j+1);
  }
}

% boundary of the computational box
\draw[thin, black!60] (-\N,-\N) rectangle (\N,\N);

% --------------------------------
% mark enriched nodes:
% vertices of triangles intersected by Gamma
% Criterion:
% a triangle is cut if its vertices have level-set
% values with opposite signs.
% --------------------------------
\foreach \i in {-7,...,6} {
  \foreach \j in {-7,...,6} {

    % four vertices of the square cell
    \pgfmathsetmacro{\xA}{\i}
    \pgfmathsetmacro{\yA}{\j}

    \pgfmathsetmacro{\xB}{\i+1}
    \pgfmathsetmacro{\yB}{\j}

    \pgfmathsetmacro{\xC}{\i+1}
    \pgfmathsetmacro{\yC}{\j+1}

    \pgfmathsetmacro{\xD}{\i}
    \pgfmathsetmacro{\yD}{\j+1}

    % level set values phi = (x-xc)^2 + (y-yc)^2 - R^2
    \pgfmathsetmacro{\phiA}{(\xA-\xc)^2 + (\yA-\yc)^2 - \R^2}
    \pgfmathsetmacro{\phiB}{(\xB-\xc)^2 + (\yB-\yc)^2 - \R^2}
    \pgfmathsetmacro{\phiC}{(\xC-\xc)^2 + (\yC-\yc)^2 - \R^2}
    \pgfmathsetmacro{\phiD}{(\xD-\xc)^2 + (\yD-\yc)^2 - \R^2}

    % Triangle 1: A-B-C
    \pgfmathsetmacro{\maxABC}{max(max(\phiA,\phiB),\phiC)}
    \pgfmathsetmacro{\minABC}{min(min(\phiA,\phiB),\phiC)}
    \pgfmathtruncatemacro{\cutABC}{ifthenelse(\maxABC>0 && \minABC<0,1,0)}

    \ifnum\cutABC=1
      \node[green!70!black, scale=0.85] at (\xA,\yA) {$\times$};
      \node[green!70!black, scale=0.85] at (\xB,\yB) {$\times$};
      \node[green!70!black, scale=0.85] at (\xC,\yC) {$\times$};
    \fi

    % Triangle 2: A-C-D
    \pgfmathsetmacro{\maxACD}{max(max(\phiA,\phiC),\phiD)}
    \pgfmathsetmacro{\minACD}{min(min(\phiA,\phiC),\phiD)}
    \pgfmathtruncatemacro{\cutACD}{ifthenelse(\maxACD>0 && \minACD<0,1,0)}

    \ifnum\cutACD=1
      \node[green!70!black, scale=0.85] at (\xA,\yA) {$\times$};
      \node[green!70!black, scale=0.85] at (\xC,\yC) {$\times$};
      \node[green!70!black, scale=0.85] at (\xD,\yD) {$\times$};
    \fi

  }
}

% --------------------------------
% interface Gamma
% --------------------------------
\draw[blue!70!black, line width=0.3pt] (\xc,\yc) circle (\R);

% --------------------------------
% labels
% --------------------------------
\node[font=\Large] at (-0.25,0.20) {$\Omega_0$};
\node[font=\Large] at (-2.4,-5.65) {$\Omega_1$};
\node[font=\Large] at (4.95,-0.25) {$\Gamma$};

\end{tikzpicture}
\caption{The mesh and enriched nodes for interface problems.}
\label{fig:int}
\end{figure}
% From the trace and embedding theorem, this assumption implies that the solution u belongs to $H^{2+\epsilon}(\Omega_0\cup\Omega_1)$ in 2D, where $\epsilon>0$.

In the remainder of this section, we revisit this construction and present a new proof that obviates the need for additional regularity while maintaining the optimal convergence rate. We emphasize that while this analysis is conducted within the standard SGFEM framework—typically involving a single, fixed enrichment function—it yields critical insights into our neural network architecture. Specifically, it elucidates the rationale behind adopting the formulation in \eqref{interface_neural} for neural enrichment.

Let $\Gamma_i = \Gamma\cap\omega_i$ denote the part of $\Gamma$ in patch $\omega_i$ associated with point $P_i$. Let $\omega_i^r = \omega_i\cap\bar\Omega_r, r=0, 1.$ The following lemma is valid.
\begin{lemma}
    Let $u$ be the solution of \eqref{model2}, then following patch wised estimates hold
    \begin{align}
        &\Vert [\partial_nu]\Vert_{H^{1/2}(\Gamma_i)}\le C\left(\vert u_0\vert_{H^2(\omega_i^0)}+\vert u_1\vert_{H^2(\omega_i^1)}\right)\label{Ju1/2}\\
        &\Vert [\partial_n u]\Vert_{L^2(\Gamma_i)}\le Ch^{1/2}\left(\vert u_0\vert_{H^2(\omega_i^0)}+\vert u_1\vert_{H^2(\omega_i^1)}\right)\label{JuL2}
    \end{align}
    \label{lemma1}
\end{lemma}
\begin{proof}
    We only prove the first inequality here, and the other one could be derived in a similar way. By scaling argument, it suffices to show
    \begin{equation*}
            \Vert [\partial_n \hat{u}]\Vert_{H^{1/2}(\hat\Gamma)}\le C\left(\vert \hat u_0\vert_{H^2(\hat\omega^0)}+\vert \hat u_1\vert_{H^2(\hat\omega^1)}\right)
    \end{equation*}
    on the reference patch. From the trace theorem,
    \begin{equation*}
    \Vert [\partial_n \hat{u}]\Vert_{H^{1/2}(\hat\Gamma)}\le\Vert\partial_n\hat u_0\Vert_{H^{1/2}(\hat\Gamma)}+\Vert\partial_n\hat u_1\Vert_{H^{1/2}(\hat\Gamma)}\le C\left(\Vert \hat u_0\Vert_{H^2(\hat\omega^0)}+\Vert \hat u_1\Vert_{H^2(\hat\omega^1)}\right).
    \end{equation*}
    For any $p\in\mathbb P_1(\hat\omega)$, we have $\Vert [\partial_n \hat{u}]\Vert_{H^{1/2}(\hat\Gamma)}=\Vert [\partial_n (\hat{u}+p)]\Vert_{H^{1/2}(\hat\Gamma)}$.
    Thus, from the equivalent results in the quotient space, one immediately has
    \begin{equation*}
    \begin{aligned}
    \Vert [\partial_n \hat{u}]\Vert_{H^{1/2}(\hat\Gamma)}=\Vert[\partial_n \hat{u}]\Vert_{H^{1/2}/\mathbb P_1(\hat\Gamma)}&\le C\left(\Vert \hat u_0\Vert_{H^2/\mathbb P_1(\hat\omega^0)}+\Vert \hat u_1\Vert_{H^2/\mathbb P_1(\hat\omega^1)}\right)\\
    &\le C\left(\vert \hat u_0\vert_{H^2(\hat\omega^0)}+\vert \hat u_1\vert_{H^2(\hat\omega^1)}\right).
    \end{aligned}
    \end{equation*}
\end{proof}

We next present Lemma \ref{lemma4.2}, which is crucial for deriving the optimal convergence order.
This result is closely related to Lemma~4.2 in~\cite{wang2025general}, but differs in the regularity
assumptions imposed on the solution. This difference arises from the specific construction of the coefficients $\sigma_i$ \eqref{special_sigma}, which avoids the requirement of nodal values. 
\begin{lemma}
    Let $u$ be the solution of \eqref{model2}. For enrichment function (quasi-distance function) $\phi$ that satisfies the general enrichment conditions \eqref{quasi}, there exists a constant $\sigma_i$ and $z_i\in\mathbb P_1(\omega_i)$ for every patch $\omega_i$, such that
    \begin{equation}
    \left.
    \begin{aligned}
    \Vert u-z_i-\sigma_i\phi\Vert^2_{H^l(\omega_i)}\\
        \Vert [\mathcal{I}-\mathcal{I}_h](u-z_i-\sigma_i\phi)\Vert^2_{H^l(\omega_i)}
    \end{aligned}
    \right\}
         \le Ch^{4-2l}\left(\vert u_0\vert^2_{H^2(\omega_i^0)}+\vert u_1\vert^2_{H^2(\omega_i^1)}\right), l=0, 1.
    \end{equation}
    \label{lemma4.2}
\end{lemma}
\begin{proof}
    We first look for the constant $\sigma_i$. Let
    \begin{equation}
        \sigma_i = \frac{\int_{\Gamma_i}[\partial_n u]ds}{\int_{\Gamma_i}\vert[\partial_n \phi]\vert ds},
        \label{special_sigma}
    \end{equation}
    $w_i=u-\sigma_i\phi$, and $\psi_i(s):=[\partial_n w_i]$. Obviously, $\psi_i\in H^{1/2}(\Gamma_i)$. By applying \eqref{JuL2}, the Cauchy-Schwarz inequality, and the general enrichment conditions \eqref{quasi}, we have the following estimate for $\sigma_i$
    \begin{equation}
        \vert\sigma_i\vert\le\frac{h^{1/2}\Vert [\partial_n u]\Vert_{L^2(\Gamma_i)}}{\kappa_1h}\le C\left(\vert u_0\vert_{H^2(\omega_i^0)}+\vert u_1\vert_{H^2(\omega_i^1)}\right).
        \label{sigma}
    \end{equation}
    Thus, by employing Lemma~\ref{lemma1} again along with \eqref{sigma}, one obtains
    \begin{equation}
        \Vert\psi_i\Vert_{H^{1/2}(\Gamma_i)}\le\Vert [\partial_n u]\Vert_{H^{1/2}(\Gamma_i)}+\vert\sigma_i\vert\Vert [\partial_n \phi]\Vert_{H^{1/2}(\Gamma_i)}\le C\left(\vert u_0\vert_{H^2(\omega_i^0)}+\vert u_1\vert_{H^2(\omega_i^1)}\right).
        \label{psi_estimate}
    \end{equation}
    
    Next, we decompose $w_i$ into two parts as $w_i = w_i^s + w_i^j$, where $w_i^s$ represents a smooth part, while $w_i^j$ denotes a term involving  flux discontinuity. To define $w_i^s$, we lift $\psi_i$ from $H^{1/2}(\Gamma_i)$ to $H^1_0(\omega_i)\cap H_0^2(\omega^0_i\cup\omega^1_i)$. In fact, we can consider an elliptic interface problem on the reference patch $\hat\omega$ as follows:
    \begin{equation}
    \begin{aligned}
        &-\Delta \hat w_i^j = 0\quad \text{in }\hat\omega, \qquad
        [\hat w_i^j]|_{\hat\Gamma_i}=0,\qquad
        [\partial_n\hat{w}_i^j]|_{\hat\Gamma_i}=\hat\psi_i, \qquad\hat w_i^j\left|_{\partial\hat\omega}\right.=0.
    \end{aligned}
    \end{equation}
    From the elliptic regularity results, $\hat w_i^j$ could be bounded as follows:
    \begin{equation}
        \Vert \hat w_i^j\Vert_{H^2(\hat\omega^0\cup\hat\omega^1)}\le C\Vert \hat\psi_i\Vert_{H^{1/2}(\hat\Gamma_i)}.
    \end{equation}
    By mapping everything back into the current patch $\omega_i$, we have
    \begin{equation}
        \Vert w_i^j\Vert_{H^l(\omega^0_i\cup\omega_i^1)}\le Ch^{2-l}\Vert\psi_i\Vert_{H^{1/2}(\Gamma_i)}, l=0, 1, 2.
        \label{lift}
    \end{equation}
    Combining with \eqref{psi_estimate}, one would derive
    \begin{equation}
        \Vert w_i^j\Vert_{H^l(\omega^0_i\cup\omega_i^1)}\le Ch^{2-l}\left(\vert u_0\vert_{H^2(\omega_i^0)}+\vert u_1\vert_{H^2(\omega_i^1)}\right), l=0, 1, 2.
        \label{jump}
    \end{equation}

    Let $w_i^s = w_i - w_i^j$. It is easy to verify that $w_i^s\in H^2(\omega_i)$. We can therefore define the Lagrangian interpolation $z_i = \mathcal{I}_hw_i^s$ for this smooth part. The standard interpolation error estimates yield
    \begin{equation}
    \begin{aligned}
        \Vert w_i^s-z_i\Vert^2_{H^l(\omega)}&\le Ch^{4-2l}\vert w_i^s\vert^2_{H^2(\omega_i)}, l=0,1\\
        &\le Ch^{4-2l}\left(\vert u\vert^2_{H^2(\omega_i^0\cup\omega_i^1)}+\vert\sigma_i\vert^2\vert \phi\vert^2_{H^2(\omega^0_i\cup\omega^1_i)}+\vert w_i^j\vert^2_{H^2(\omega^0_i\cup\omega^1_i)}\right)\\
        &\le Ch^{4-2l}\left(\vert u_0\vert^2_{H^2(\omega_i^0)}+\vert u_1\vert^2_{H^2(\omega_i^1)}\right),\; l=0, 1.
    \end{aligned}
        \label{interpolation}
    \end{equation}
    Combining \eqref{jump} and \eqref{interpolation}, we immediately obtain the first term.
    \begin{equation}
        \Vert u-z_i-\sigma_i\phi\Vert^2_{H^l(\omega_i)}\le Ch^{4-2l}\left(\vert u_0\vert^2_{H^2(\omega_i^0)}+\vert u_1\vert^2_{H^2(\omega_i^1)}\right), \; l=0, 1.
    \end{equation}

    We next estimate the second term in \eqref{lemma4.2}. 
    % Let $\mathcal{I}$ be the identity operator. We have
    % \begin{equation}
    %     \Vert \mathcal{I}_h(u-z_i-\sigma_iG)\Vert^2_{H^l(\omega_i)} \le \Vert (\mathcal{I}-\mathcal{I}_h)(u-z_i-\sigma_iG)\Vert^2_{H^l(\omega_i)} + \Vert u-z_i-\sigma_iG\Vert^2_{H^l(\omega_i)}
    % \end{equation}
    % The second term we have already estimated. For the first term, 
    Since $z_i - \mathcal{I}z_i=0$, we have 
    \begin{equation}
        \Vert [\mathcal{I}-\mathcal{I}_h](u-z_i-\sigma_i\phi)\Vert^2_{H^l(\omega_i)}
        \le\Vert(\mathcal{I}-\mathcal{I}_h)(w_i^s)\Vert_{H^l(\omega_i)}+\Vert(\mathcal{I}-\mathcal{I}_h)(w_i^j)\Vert_{H^l(\omega_i)}
    \end{equation}
    The first term is related to standard interpolation error estimates. The second term, since $w_i^j$ has zero trace on patch $\omega_i$, degenerates into $\Vert w^j_i\Vert_{H^l(\omega_i)}, l=0, 1$. Therefore, by applying \eqref{jump} we ultimately have
    \begin{equation}
        \Vert [\mathcal{I}-\mathcal{I}_h](u-z_i-\sigma_i\phi)\Vert^2_{H^l(\omega_i)}\le Ch^{4-2l}\left(\vert u_0\vert^2_{H^2(\omega_i^0)}+\vert u_1\vert^2_{H^2(\omega_i^1)}\right), l=0, 1.
    \end{equation}
\end{proof}
% We give the main theorem of error estimates below:
\begin{theorem}
    Suppose that $u$ is the solution of the interface problem \eqref{model2}, and $u_{SG,h}$ is the SGFEM solution, the enrichment function $\phi$ satisfies the general enrichment conditions \eqref{quasi}. Then, there exists $C>0$ independent of $h$ such that
    \begin{equation}
        \Vert u-u_{SG,h}\Vert_E\le Ch\left(\vert u_0\vert_{H^2(\Omega^0)}+\vert u_1\vert_{H^2(\Omega^1)}\right).
    \end{equation}
    \label{main_theorem}
\end{theorem}
\begin{proof}
    Let 
    $v = \mathcal{I}_hu+\sum_{i\in I^{enr}_h}\sigma_iL_i(\phi-\mathcal{I}_h\phi),
    $
    where $\sigma_i$ are the constants in Lemma~\ref{lemma4.2}. We have
    \begin{equation}
    \begin{aligned}
        u - v &= \sum_{i\in I_h}L_i(u-\mathcal{I}_hu) - \sum_{i\in I^{enr}_h}\sigma_iL_i(\phi-\mathcal{I}_h\phi)\\
        &=\sum_{i\in I_h^{std}}L_i(u-\mathcal{I}_hu)+\sum_{i\in I^{enr}_h}L_i[\mathcal{I}-\mathcal{I}_h](u-z_i-\sigma_i\phi):=T_1 + T_2,
    \end{aligned}
    \end{equation}
    where $z_i$ is the polynomial in Lemma~\ref{lemma4.2}. As $I^{enr}_h=\{i\in I_h: P_i\in e_s \text{ where } e_s\cap\Gamma\neq\emptyset\}$, the patch $\omega_i$ either $\omega_i\subset\Omega_0$ or $\omega_i\subset\Omega_1$. Thus, by using the standard interpolation result and \eqref{partition}, $T_1$ is estimated as follows:
    \begin{equation}
    \begin{aligned}
        &\Vert T_1\Vert_E^2\le C\sum_{i\in I_h^{enr}}\left(\Vert L_i\Vert_{L^\infty(\Omega)}^2\vert u-\mathcal{I}_hu\vert^2_{H^1(\omega_i)}+\Vert\nabla L_i\Vert_{L^\infty(\Omega)}^2\Vert u-\mathcal{I}_hu\Vert_{L^2(\omega_i)}^2\right)\\
        &\le C\sum_{i\in I_h^{std}}\left(\vert u-\mathcal{I}_hu\vert^2_{H^1(\omega_i)}+h^{-2}\Vert u-\mathcal{I}_hu\Vert_{L^2(\omega_i)}^2\right)\le Ch^2\left(\vert u_0\vert^2_{H^2(\Omega^0)}+\vert u_1\vert^2_{H^2(\Omega^1)}\right)
        \end{aligned}
        \label{T1}
    \end{equation}
    Similarly, based on Lemma~\ref{lemma4.2}, $T_2$ yields
    \begin{equation}
        \Vert T_2\Vert^2_E\le Ch^2\left(\vert u_0\vert^2_{H^2(\Omega^0)}+\vert u_1\vert^2_{H^2(\Omega^1)}\right).
        \label{T2}
    \end{equation}
    Combining \eqref{T1} and \eqref{T2}, we have
    \begin{equation}
        \Vert u-v\Vert^2_E\le Ch^2\left(\vert u_0\vert^2_{H^2(\omega_i^0)}+\vert u_1\vert^2_{H^2(\omega_i^1)}\right).
    \end{equation}
    From  Céa's lemma we immediately have the desired results.
\end{proof}

Based on Theorem~\ref{main_theorem}, we revisit the network design defined in $\eqref{interface_neural}$. Assuming a uniform architecture and fixed random distribution initialization are adopted for all networks, the following corollary holds at the initialized state, notwithstanding their independent and random initializations:
\begin{corollary}
    Suppose that $u$ is the solution of the interface problem \eqref{model2}, and $u_{NE}$ is the initial (without training) NEFEM solution using the enrichment function \eqref{interface_neural} and \eqref{quasi}. Assume that $\partial_D\mathcal{N}_i \neq0$ for $x\in\Gamma$. Then, there exists $C>0$ independent of $h$ such that
    \begin{equation}
        \mathbb{E}_\theta[\Vert u-u_{NE}\Vert_E]\le Ch\left(\vert u_0\vert_{H^2(\Omega^0)}+\vert u_1\vert_{H^2(\Omega^1)}\right),
    \end{equation}
    where $\mathbb{E}_\theta$ is the expectation with respect to the initialized network parameters.
    \label{corollary}
\end{corollary}
\begin{proof}
    Based on Theorem~\ref{main_theorem}, we have
    \begin{equation*}
        \Vert u-u_{NE}\Vert_E\le C(\theta)h\left(\vert u_0\vert_{H^2(\Omega^0)}+\vert u_1\vert_{H^2(\Omega^1)}\right),
    \end{equation*}
    where $C(\theta)$ is a random variable drawn from a fixed distribution. The result follows directly by taking its expectation.
\end{proof}
% \begin{remark}
% Corollary~\ref{corollary} implies that the training process can be proceed from a more accurate baseline through mesh refinement.
% \end{remark}
\section{Numerical results}
\label{sec:results}
In this section, we provide numerical examples to validate the efficiency and accuracy of our proposed approach. We construct all NNs in PyTorch, and apply the Adam optimizer for training. To accelerate convergence, we employ an adaptive activation function inspired by the work of Karniadakis et al. \cite{jagtap2020adaptive}, which is defined as follows:
\begin{equation}
\sigma(na\mathcal{L}_k(z^{k-1})),
\end{equation}
where $n\ge1$ is a scale factor, $a$ is a trainable parameter, $\mathcal{L}_k$ denotes the affine transformation of layer $k$, and $\sigma$ is the common activation function, e.g., $\sin$ or $\operatorname{Tanh}$. By modifying the slope of the activation function, this parameter enables the enrichment basis to capture high-frequency information. For the final output layer $\mathcal{L}_{-1}$, we retain only the weights $w_{-1}$ and omit the bias term $b_{-1}$, as the constant shift is inherently captured by the standard FE basis. In terms of integration, we apply a 78-point Gaussian quadrature rule in each triangular element, which achieves an algebraic precision of degree 20. The coordinates and weights are sourced directly from Liu and Liu \cite{liu2024symmetric}. All numerical experiments were conducted using an NVIDIA H100 GPU and an AMD EPYC 7773X CPU.

\subsection{Example 1: A Problem with oscillated coefficients}
In this test case, we consider the following problem on $\Omega=(0, 1)^2$ with oscillated coefficient $a_\epsilon$:
\begin{equation}
\begin{aligned}
    &-\nabla\cdot \big(a_\epsilon\nabla u\big) =-1, \quad
     a_\epsilon(x,y) = \frac{1}{(2 + P\sin(2\pi x/\epsilon))(2 + P\sin(2\pi y/\epsilon))},
\end{aligned}
\label{ex1}
\end{equation}
where $P=1.5$ and $\epsilon=0.02$. Homogeneous Dirichlet conditions $ u|_{\partial\Omega} = 0$ are applied. 
% The oscillation behavior of $a_\epsilon$ is illustrated in Figure \ref{fig:aepsilon}.
In the NEFEM implementation, we enriched all interior nodes and trained the network via Algorithm \ref{algo:NeFEM} for 60 epochs and 200 epochs, respectively, with a learning rate of 0.001. For each enrichment function, we employ a compact neural network with a $[2, 20, 20, 1]$ architecture and sine activation functions. The scale factor are set to $n_1=150$ for the first hidden layer, and $n_2=2$ for the second hidden layer. The numerical performance is evaluated using the $L^2$-error, denoted by $e_{L^2}$, and the $H^1$-semi-norm error, denoted by $e_{H^1}$. As problem \eqref{ex1} lacks an analytical solution, we adopt the FEM solution obtained on a highly refined mesh ($h=1/2048$) as the reference. 
% \begin{figure}[htpb]
%     \centering
%     \includegraphics[width=0.6\linewidth]{figure/aepsilon.png}
% \caption{Left: Profile of $a_\epsilon$ for Example 1 with $\epsilon=0.02$. Right: A zoomed-in top view on the region $[0.5, 0.6]\times[0.5, 0.6]$.}
% \label{fig:aepsilon}
% \end{figure}

Table \ref{tab:nefem_fenics} compares the results of NEFEM with those of the standard FEM computed using FEniCS \cite{AlnaesEtal2015, AlnaesEtal2014}. On a relatively coarse $32\times32$ mesh with only 2050 DoFs, NEFEM already attains better approximations than that of the standard FEM on a much finer $512\times512$ mesh with 263,169 DoFs in both $L^2$-norm and $H^1$-semi-norm. In addition, the computational time of NEFEM is less than half of that required by FEM. Moreover, increasing the number of training epochs further improves the accuracy of NEFEM. With $200$ epochs, the $L^2$-error is more than one order of magnitude lower than that of the standard FEM.
\begin{table}[htbp]
  \centering
\caption{Comparison of NEFEM and FEM results generated from FEniCS. The errors $e_{L^2}$, $e_{H^1}$, and computational times are averaged over 6 independent runs.}
  \label{tab:nefem_fenics}
  \begin{tabular}{c c c c c c c}
    \toprule
           & Mesh & DoFs & Epochs &$e_{L^2}$ & $e_{H^1}$ & Time (s) \\
    \midrule
    NEFEM  & $32\times 32$   & 2\,050 & 60 & $6.75\cdot 10^{-4}$& $4.54\cdot 10^{-2}$& 2.48  \\
    FEM & $512\times 512$ & 263\,169 & -- &$1.06\cdot 10^{-3}$ & $4.66\cdot 10^{-2}$ & 6.48 \\
    NEFEM  & $32\times 32$   & 2\,050 & 200 & $1.06\cdot 10^{-4}$ & $2.46\cdot 10^{-2}$ & 7.96  \\
    \bottomrule
  \end{tabular}
\end{table}

Figure~\ref{fig:error1}(a) depicts the error evolution during the training across two mesh levels. All experiments are based on a same network structure and learning rate. It is worth emphasizing that the proposed method does not require a large number of training epochs. In our experiments, highly accurate results are typically obtained within $200$ epochs. Furthermore, let $H=DAD$, where $D$ is a diagonal matrix with $D_{ii}=A_{ii}^{-1/2}$. Define the scaled condition number $\mathfrak{R}(A)$ of $A$ by $\mathfrak{R}(A):=\kappa_2(H)$, where $\kappa_2(H)=\Vert H\Vert_2\Vert H^{-1}\Vert_2$ is the condition number of $H$ based on the $\Vert\cdot\Vert_2$ vector norm. As noted in \cite{babuvska2012stable}, the scaled condition number is a standard metric in the SGFEM literature. We employ the scaled condition number to numerically evaluate the linear independence of the enrichment functions. Figure \ref{fig:error1}(b) presents the scaled condition number evolution of NEFEM during the training process. The overall condition numbers scale as $\mathcal{O}(h^{-2})$, and remain stable throughout training, indicating that the enrichment functions are linearly independent.
\begin{figure}[htbp]
    \centering
    \begin{minipage}{0.48\textwidth}
        \centering
        \begin{overpic}[width=0.7\linewidth]{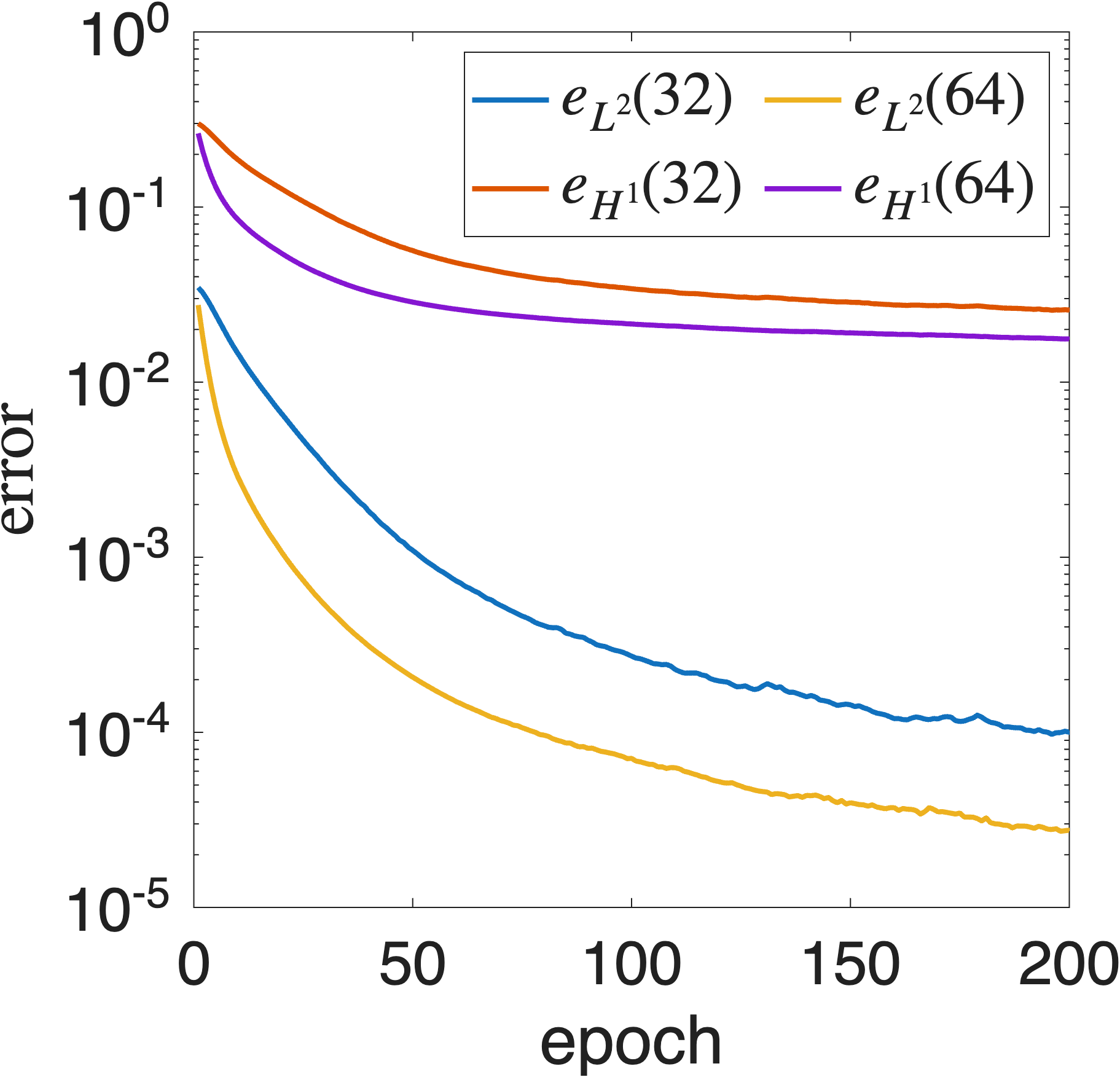}
            \put(-10,90){(a)}
        \end{overpic}
    \end{minipage}
    \hfill
    \begin{minipage}{0.48\textwidth}
        \centering
        \begin{overpic}[width=0.7\linewidth]{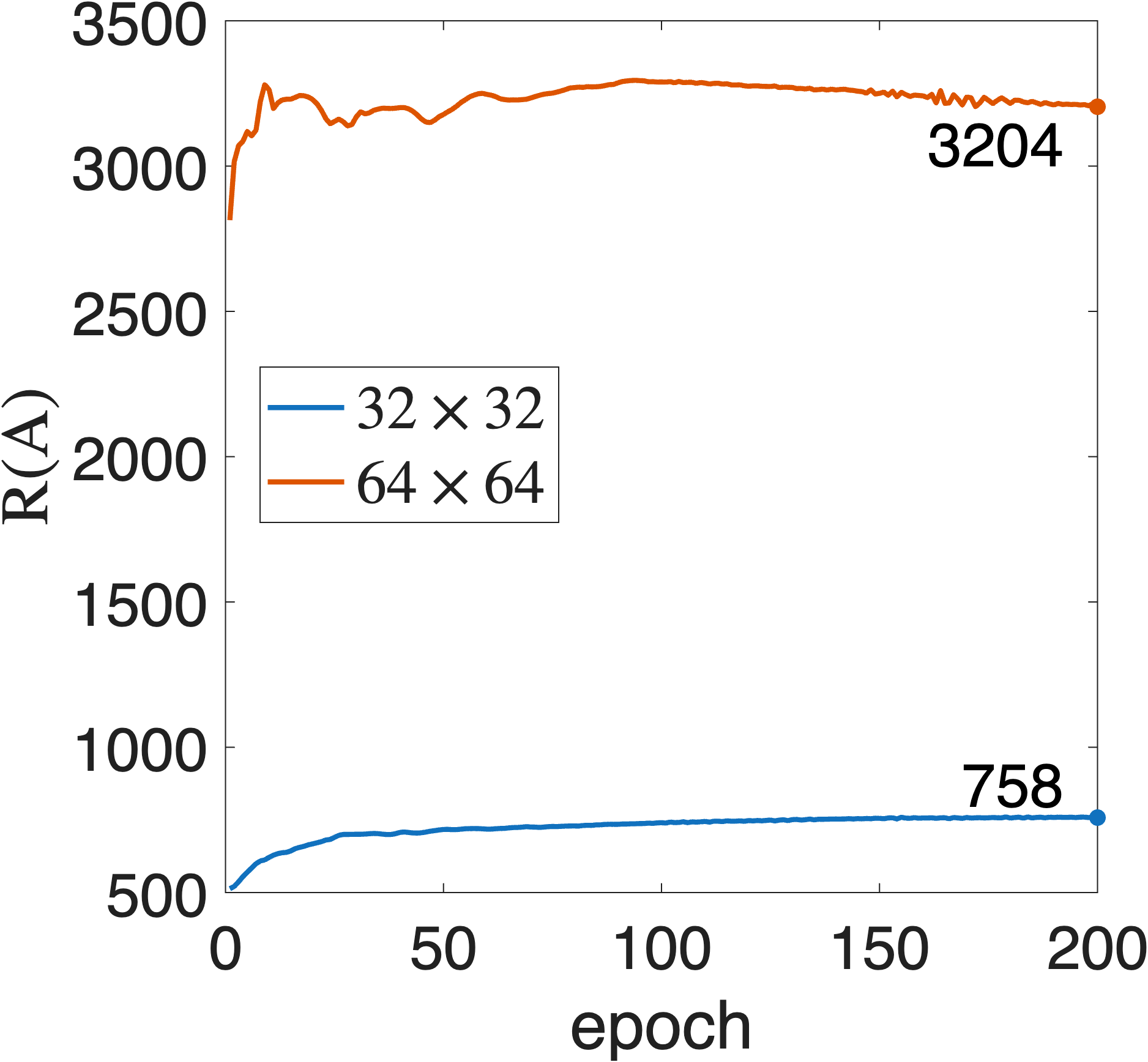}
            \put(-10,90){(b)}
        \end{overpic}
    \end{minipage}
    \caption{Evolution of the error (a) and the scaled condition number (b) during training on two different mesh levels. The final values of $\mathfrak{R}(A)$ are annotated in the figure.}
    \label{fig:error1}
\end{figure}
\subsection{Example 2: A problem with local oscillations}
This test case presents the numerical performance of the error estimator and the adaptivity Algorithm \ref{adaptive}. We consider the following exact solution, which oscillates locally
    \begin{equation}
    \begin{aligned}
        u(x,y) = \big( \sin(2\pi x) +  g(x)\sin(50\pi(x - 0.5)) \big)
        \big( \sin(2\pi y) + g(y)\sin(50\pi(y - 0.5)) \big),
    \end{aligned}
    \label{ex 2}
    \end{equation}
where $g(x)=e^{-100(x - 0.5)^2}$. The diffusion coefficient $a$ is set as one. The source term $f$ is calculated from the analytical solution. As shown in Figure \ref{fig:ex2} (a), \eqref{ex 2} is smooth at the four corners but exhibits oscillatory behavior within a cross-shaped region. We employ the estimator to select enriched nodes, and train the network via Algorithm \ref{adaptive} for 200 epochs with a learning rate of 0.001. The marking parameters are set to $\alpha_1=\alpha_2=0.6$, and the adaptive epochs are set to $H_1=H_2=50$. For each enrichment function, we employ the same network architecture and activation functions as those in Example 1.
\begin{figure}[htbp]
    \centering
    \begin{minipage}{0.45\textwidth}
        \centering
        \begin{overpic}[width=0.8\linewidth]{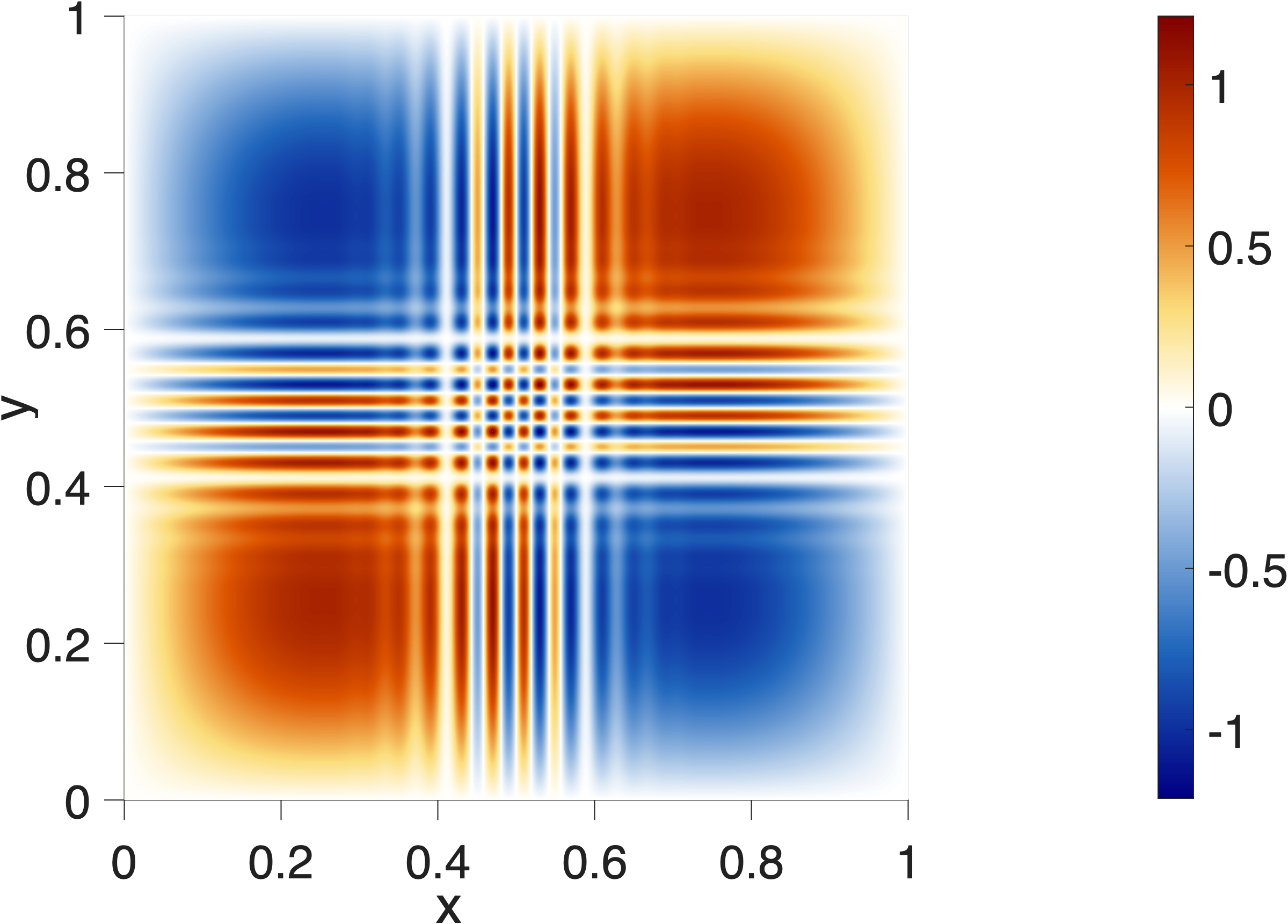}
            \put(-10,80){(a)}
        \end{overpic}
    \end{minipage}
    \hfill
    \begin{minipage}{0.45\textwidth}
        \centering
        \begin{overpic}[width=0.8\linewidth]{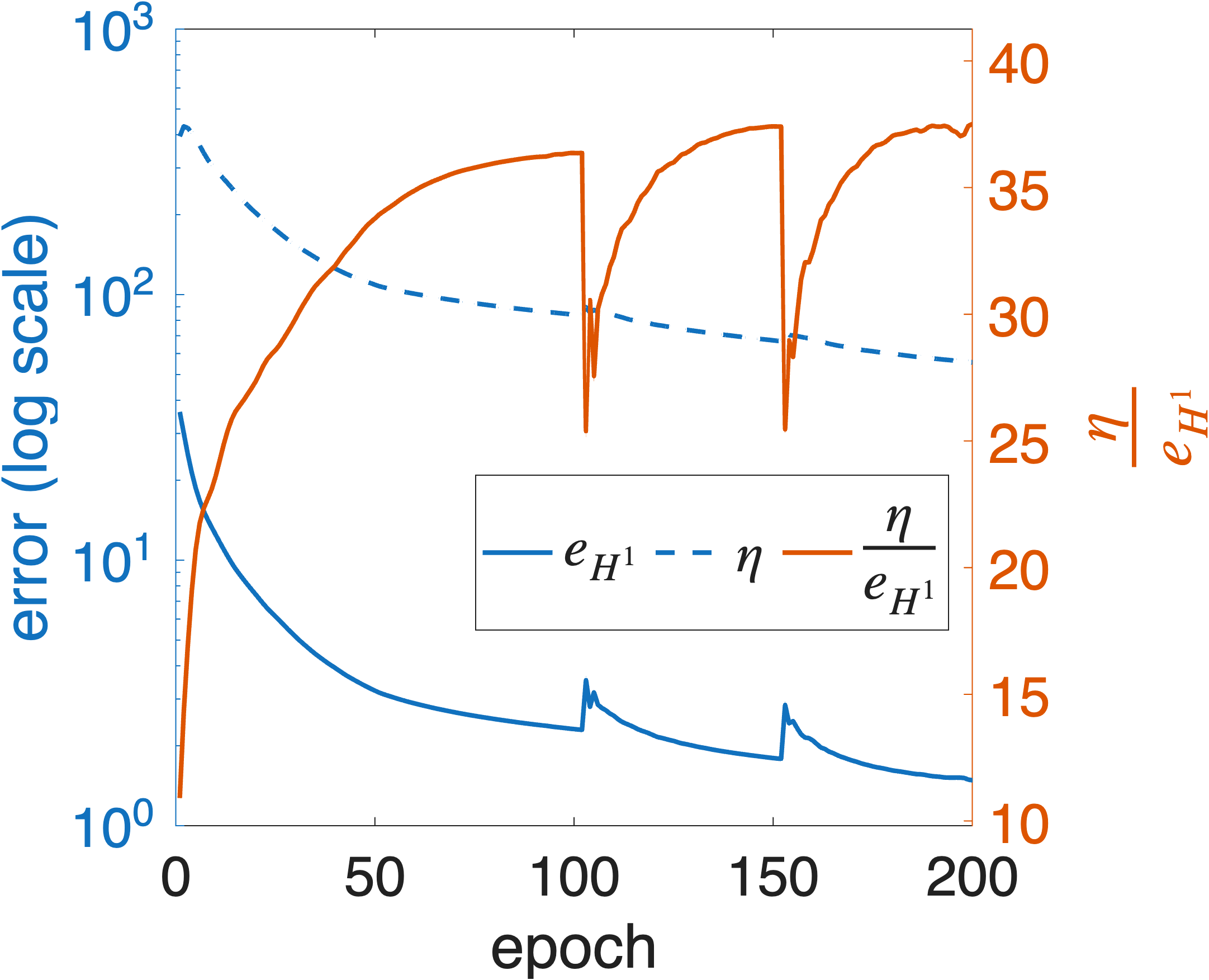}
            \put(-10,80){(b)}
        \end{overpic}
    \end{minipage}
    \caption{(a) Profile of $u$. (b) Evolution of $e_{H^1}$ and the estimator on a $32\times32$ mesh.}
    \label{fig:ex2}
\end{figure}
\begin{figure}[htbp]
    \centering
    \begin{minipage}{0.32\textwidth}
        \centering
        \begin{overpic}[width=\linewidth]{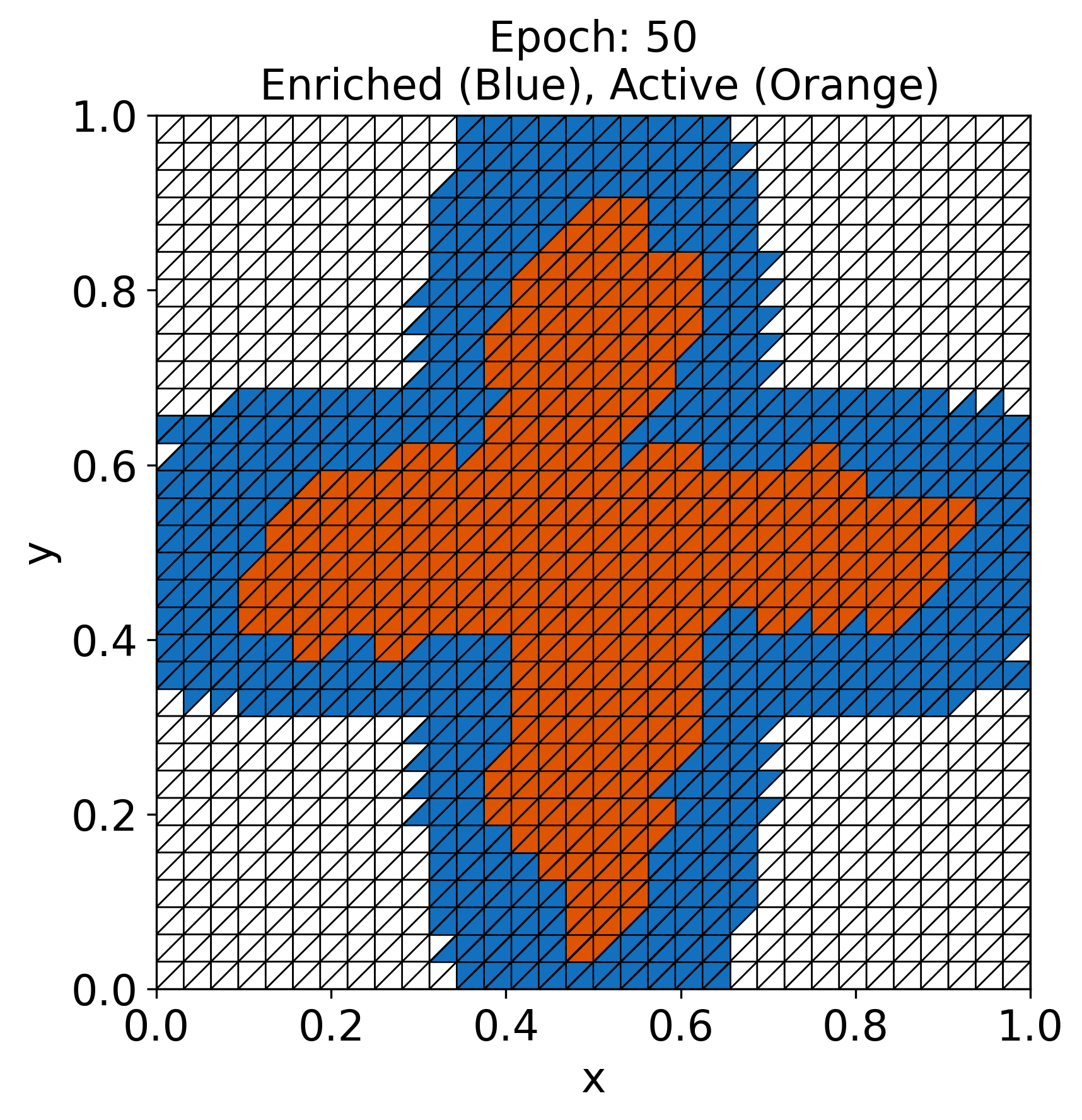}
            \put(-10,90){(a)}
        \end{overpic}
    \end{minipage}
    \hfill
    \begin{minipage}{0.32\textwidth}
        \centering
        \begin{overpic}[width=\linewidth]{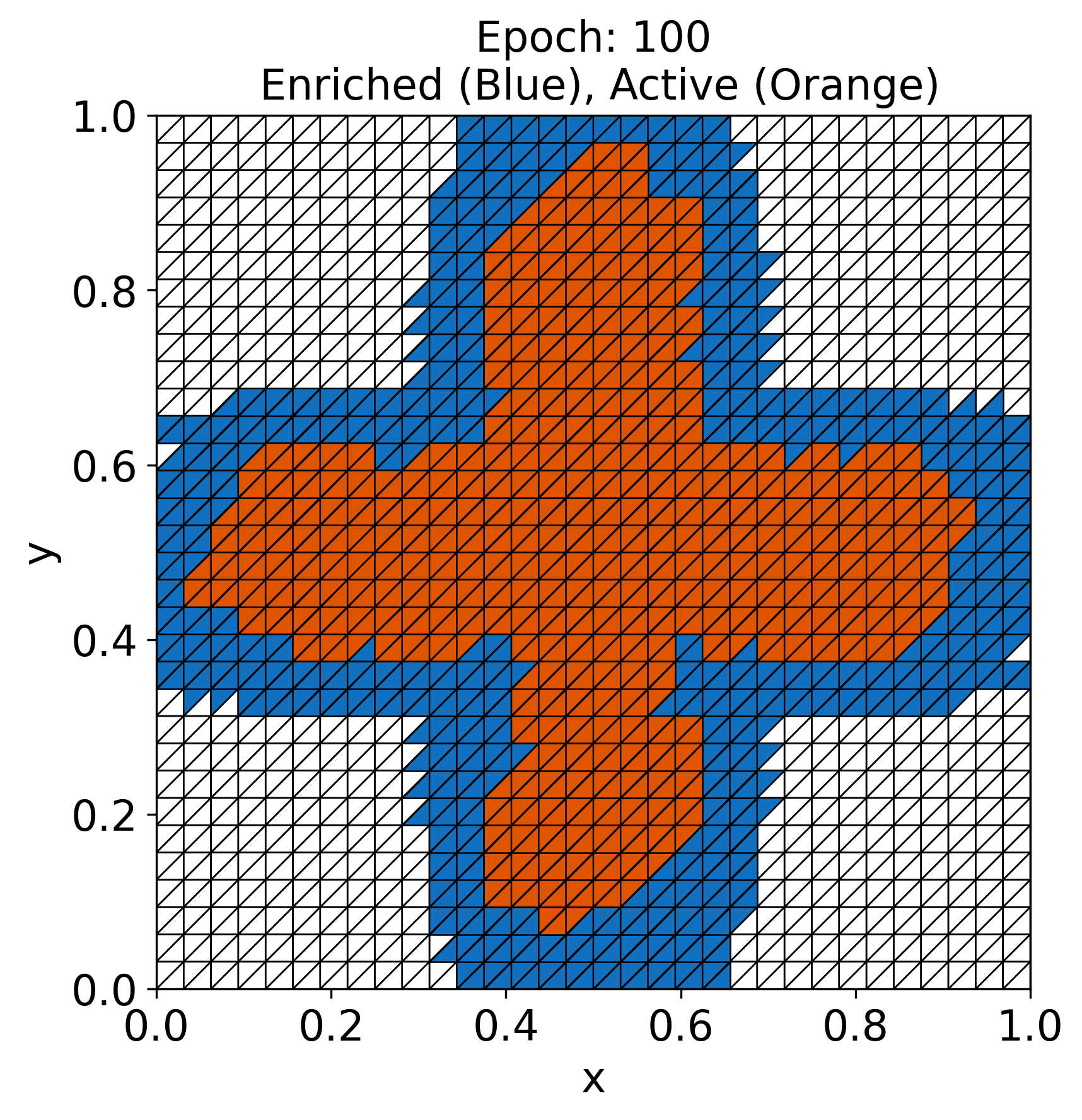}
            \put(-10,90){(b)}
        \end{overpic}
    \end{minipage}
    \hfill
    \begin{minipage}{0.32\textwidth}
        \centering
        \begin{overpic}[width=\linewidth]{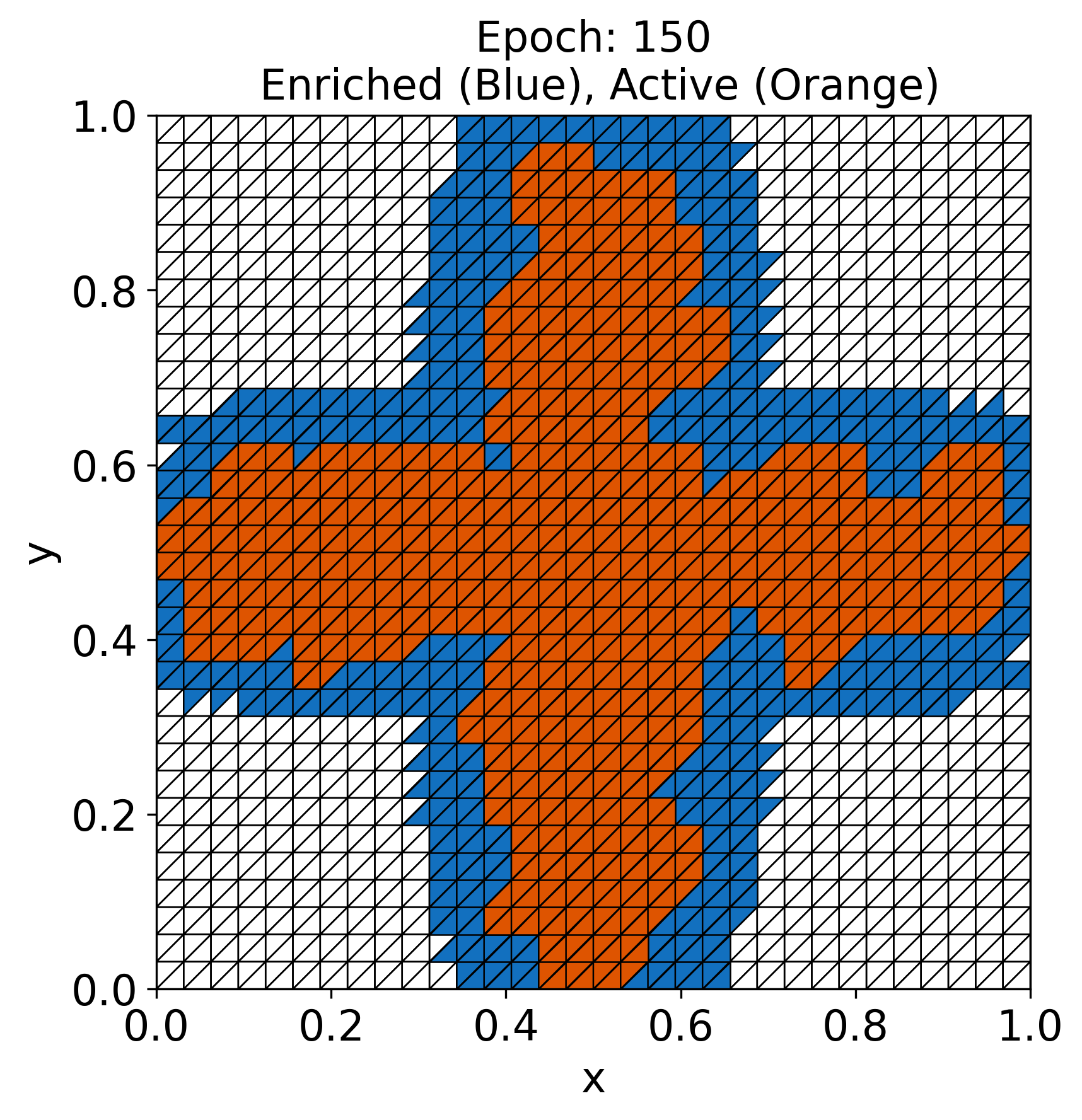}
            \put(-10,90){(c)}
        \end{overpic}
    \end{minipage}
    \caption{Distributions of enriched (blue) and active (orange) nodes at epochs 50, 100, and 150, respectively.}
    \label{fig:local}
\end{figure}

Figure \ref{fig:ex2} (b) illustrates the evolution of $e_{H^1}$ and the estimator $\eta=\sum_K\eta_K$ during training. As discussed in Section \ref{sec:error estimates}, the performance of the estimator dependents on the updated enrichment functions. To evaluate its robustness, we track the evolution of the effectivity index $\eta/e_{H^1}$ during training. A strongly fluctuated effectiveness index is undesirable, as it indicates that the estimator fails to accurately track the true error. From the right axis (orange) of Figure \ref{fig:ex2} (a), while the true error $e_{H^1}$ decreases by a factor of approximately 25, the effectivity index only increases by a factor of roughly 3 and remains well-bounded. This demonstrates that the estimator is sharp and numerically robust throughout the training. In Figure \ref{fig:local} (b)-(d), we present distributions of enriched nodes and active nodes for each selections. The enriched nodes concentrate along the cross-shaped region, which aligns well with the localized oscillations of the solution. Table~\ref{tab:adaptivity} compares the results of the adaptive NEFEM with those of the standard FEM computed using FEniCS. The adaptive NEFEM achieves a better balance between computational efficiency and accuracy.
\begin{table}[htbp]
  \centering
\caption{Comparison of adaptive NEFEM and FEM results. Adaptive parameters are set to  $\alpha:=\alpha_1=\alpha_2$, $H:=H_1=H_2$. The errors $e_{L^2}$, $e_{H^1}/\vert u\vert_{H^1}$, and computational times are averaged over 6 independent runs.}
  \label{tab:adaptivity}
  \begin{tabular}{c c c c c c c c}
    \toprule
           & Mesh & Epochs & $\alpha$ & $H$& $e_{L^2}$ & $e_{H^1}/\vert u\vert_{H^1}$ & Time (s) \\
    \midrule
    NEFEM  & $32\times 32$ & 200 & 0.6& 50 &$4.58\cdot 10^{-3}$ & $3.66\cdot 10^{-2}$ & 6.42  \\
    NEFEM  & $32\times 32$ & 200& 1.0 & -- & $3.12\cdot 10^{-3}$ & $3.00\cdot 10^{-2}$ &  8.04 \\
    FEM & $512\times 512$ & -- & -- & -- &$2.84\cdot 10^{-3}$ & $9.81\cdot 10^{-2}$ & 6.89 \\
    \bottomrule
  \end{tabular}
\end{table}
\subsection{Example 3: An Interface problem}
In this example, we consider the interface problem. The square $\Omega=[-1,1]^2$ is split into a circle $\Omega_1=\mathcal{B}_R(x_m)$, where $R=0.5$ and $x_m = (0, 0.15)$ and $\Omega_2=\Omega\backslash\Omega_1$. The analytical solution is
\begin{equation}
    u(x)=\left\{\begin{aligned}
        &-2a_2\Vert x-x_m\Vert^4,&&\quad x\in\Omega_1,\\
        &-a_1\Vert x-x_m\Vert^2 + \frac{1}{4}a_1 - \frac{1}{8}a_2,&&\quad x\in\Omega_2,
    \end{aligned}
    \right.
\end{equation}
where $a_1=0.1$ and $a_2=1.0$. The right-hand side $f_i=-\nabla a_i\cdot\nabla u$ and Dirichlet boundary data are defined from the exact solution. In this case, we \textit{only} enrich those elements intersected by the interface. For each enrichment function, we employ a neural network with a $[3, 20, 20, 1]$ architecture and sine activation function. The scale factor are set to $n_1=10$ and $n_2=2$. The additional dimension of the input, as discussed beforehand, is the distance function $D(x)=\vert x-x_m\vert$. In Figure \ref{fig:interface1} (a), we plot the expectations of $L^2$ errors $\mathbb{E}_{L^2}$ and energy norm errors $\mathbb{E}_E$ obtained on several mesh levels at the initial stage. Errors on each mesh levels are averaged over 500 independent runs. This verifies the $\mathcal{O}(h)$ results in Corollary \ref{corollary}. We present the evolution of the error during training on different mesh levels in Figure \ref{fig:interface1} (b). This shows that the training process can be proceed from a more accurate baseline through mesh refinement. Note that the error decreases rapidly and then remains stable. This is because we only enrich the elements intersected by the interface.
\begin{figure}[htbp]
    \centering
    \begin{minipage}{0.48\textwidth}
        \centering
        \begin{overpic}[width=0.7\linewidth]{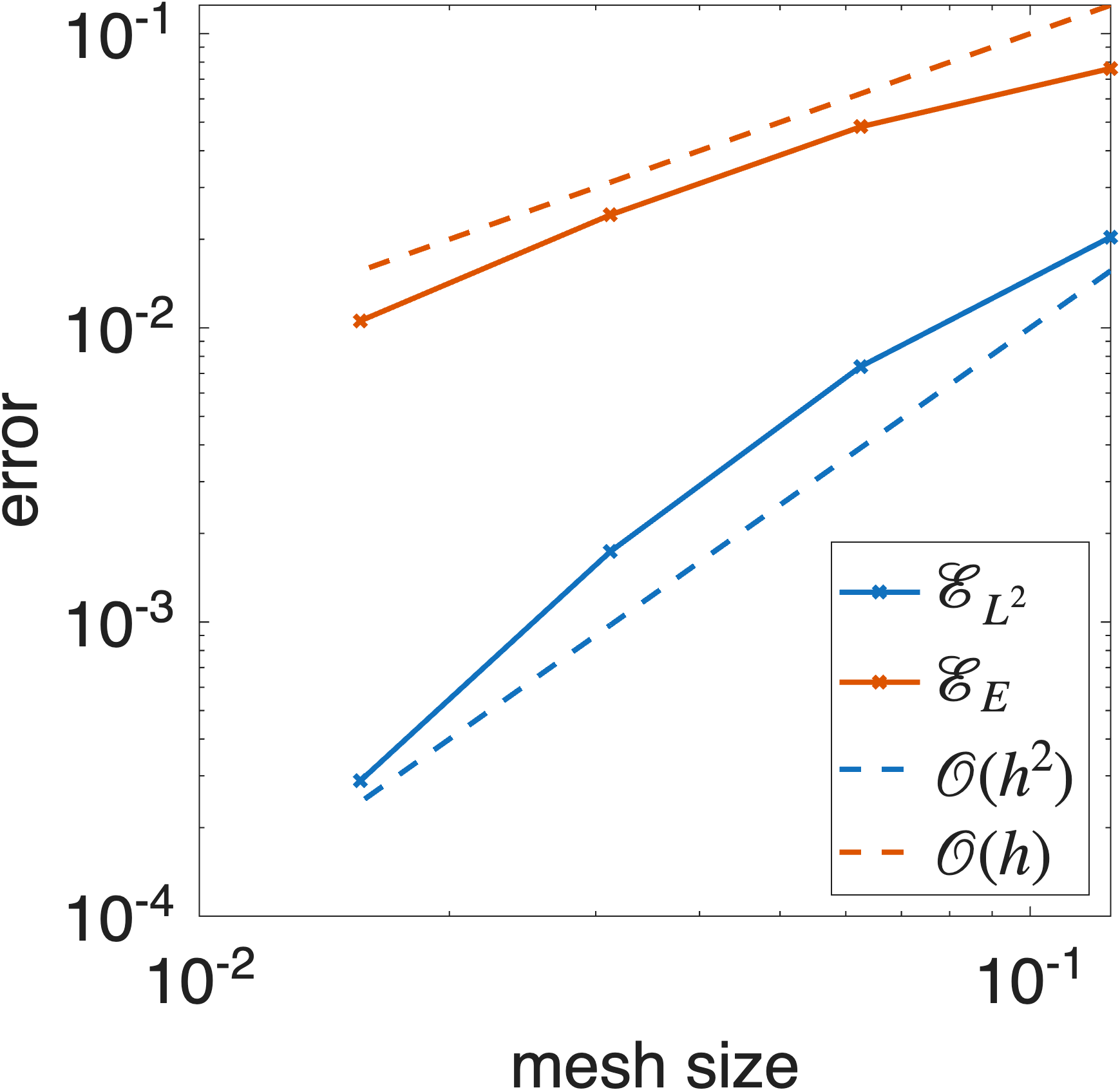}
            \put(-10,90){(a)}
        \end{overpic}
    \end{minipage}
    \hfill
    \begin{minipage}{0.48\textwidth}
        \centering
        \begin{overpic}[width=0.7\linewidth]{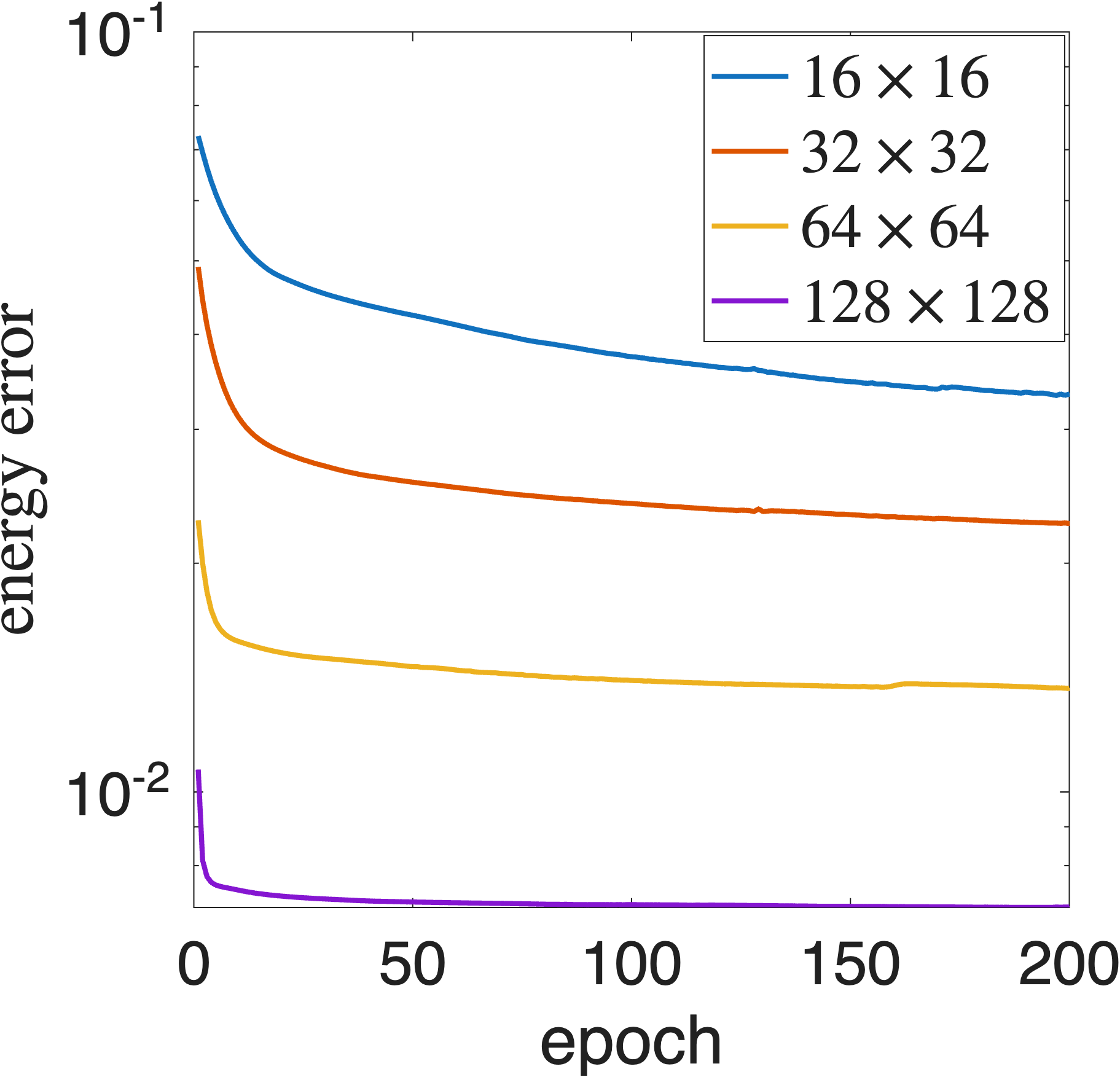}
            \put(-10,90){(b)}
        \end{overpic}
    \end{minipage}
    \caption{(a) Convergence order of the expectation error averaged over 500 independent runs at the initial stage. (b) Energy error evolution on different mesh levels.}
    \label{fig:interface1}
\end{figure}

% \subsection{Example 4: An interface problem with strong oscillations}
% In this case, we consider the analytical solution below
% \begin{equation}
% u(x,y)=
% \begin{cases}
% E + (x-0.53)E\phi, & x < 0.53, \\
% E + r(x-0.53)E\phi, & x > 0.53,
% \end{cases}
% \end{equation}
% where $E = \sin\bigl(\pi(x-0.03)\bigr)\sin(\pi y)$, $\phi = 0.05\,\sin\left(\frac{2\pi x}{\epsilon}\right)\cos\left(\frac{2\pi y}{\epsilon}\right)$ with $\epsilon=0.02$ and $r=100$. The profile of $u$ are shown in Figure \ref{}.

\section{Conclusions}
\label{sec:conclusions}
In this paper we have proposed the Neural Enrichment Finite Element Method, a hybrid framework based on SGFEM with the adaptivity of neural-network enrichment functions. The method constructs problem-dependent local approximation spaces through a Ritz-based training procedure while retaining compatibility with standard finite element solvers. For elliptic problems with localized oscillations, we introduced a residual-based adaptive strategy and established the reliability and local efficiency of the corresponding estimator. For interface problems, we showed that enrichment functions satisfying general interface conditions yield optimal energy-norm convergence without imposing additional regularity assumptions beyond piecewise $H^2$-regularity; this analysis also motivates the proposed interface-aware network architecture. The numerical experiments confirm that NEFEM can achieve substantially improved accuracy with far fewer degrees of freedom than standard FEM, while maintaining stable conditioning during training. 

Future work will extend the present framework to time-dependent problems, such as parabolic equations, and to mixed formulations, such as the Stokes problem. Another promising direction is the application of NEFEM to multiphysics models, including Stokes--Darcy systems and fluid--structure interaction problems.
% \section*{Acknowledgments}
% We would like to acknowledge the assistance of volunteers in putting
% together this example manuscript and supplement.
\appendix
\section{Inverse estimates for the standard SGFEM}
In this appendix, we establish the inverse estimates for the standard SGFEM. Given these estimates, the proof of local efficiency follows the same arguments as those for NEFEM. To this end, we impose the following assumption.
\begin{assumption}[Uniform finite order non-flatness]
\label{appen:assumption}
    Let the enrichment function satisfy $\phi\in W^{m_0+1, \infty}$ for some $m_0\ge2$. There exists a constant $c_0>0$ such that, $\forall x\in\bar\Omega$,
    \begin{equation}
        \max_{2\le m\le m_0}\Vert \nabla^m\phi(x)\Vert\ge c_0>0,
        \label{assumption_eq}
    \end{equation}
    where $\Vert \nabla^m\phi(x)\Vert$ denotes the Frobenius norm of the $m$-th order derivative tensor.
\end{assumption}
\begin{remark}
        It is satisfied by all polynomials of degree at least two. 
        Moreover, some non-polynomial functions such as $\sin(x)$ and $\exp(x)$ also fulfill this assumption. Assumption~\ref{appen:assumption} mainly excludes the degenerate case in which the enrichment function is locally close to a linear polynomial.
\end{remark}
\begin{lemma}[Inverse inequality]
\label{appendix: Inverse inequality}
Let the enrichment function $\phi$ satisfy Assumption~\ref{appen:assumption}. There exists a constant $C>0$ independent of $h$ such that 
\begin{equation}
    \Vert \nabla v_h\Vert_{L^2(K)}\le Ch^{-1}\Vert v_h\Vert_{L^2(K)}
    \label{inverse}
\end{equation}
for all $v_h\in V_h(K)=V_{FEM} \oplus \text{span}\{L_jr_{h, K}: j\in I_h^{enr}\}_{j=1}^M$.
\end{lemma}

\begin{proof}
    We employ a scaling argument by considering the function space $\hat{V}_{h, K}$ defined on the reference element. As all norms are equivalent in finite dimensional spaces, we have
    \begin{equation}
        \Vert\nabla\hat{v}\Vert_{L^2(\hat K)}\le C(h, K)\Vert \hat{v}\Vert_{L^2(\hat K)},
        \label{reference}
    \end{equation}
    for all $\hat{v}\in \hat{V}_{h, K}$. In the following, we show that $C(h, K)$ is uniformly bounded, i.e., $C(h, K)\le C^*$. 

    (Step 1) Taylor’s theorem at a fixed point $\hat{x}_K$ on $\hat K$ implies that
     \begin{equation}
         \hat{r}_{h, K} = \sum_{m=2}^{m_0} h^m\hat b_{m, K} + \hat{R}_{h, K} :=\hat\pi_{h, K} + \hat{R}_{h, K}.
     \end{equation}
     Here, each $\hat{b}_{m, K}$ depends linearly on $\{D^\alpha\phi(x_K)\}_{|\alpha|=m}$, where $x_K$ is the corresponding point in physical element $K$. Linear interpolation yields $\hat b_{m, K}\in\mathcal{B}_{m, \hat{x}_K}:= (I - \hat{\mathcal{I}}) \text{span}\{(\hat x- \hat{x}_K)^\alpha, |\alpha|=m\}$, where $I$ and $\hat{\mathcal{I}}$ are the identity operator and linear interpolation on $\hat K$, respectively. Since $\mathcal{B}_{m, \hat{x}_K}$ is finite-dimensional, and the mapping
     $\{D^\alpha\phi(x_K)\}_{|\alpha|=m}\mapsto\hat{b}_{m, K}, m\ge2$
     is linear and injective, there exist $c_1, c_2>0$,
     \begin{equation}
         \frac{c_1}{m!}\Vert\nabla^m\phi(x_K)\Vert\le\Vert\hat{b}_{m, K}\Vert_{L^2(\hat K)}\le \frac{c_2}{m!}\Vert\nabla^m\phi(x_K)\Vert.
     \end{equation}
     Moreover, the remainder satisfies
     \begin{equation}
         \Vert\hat{R}_{h, K}\Vert_{H^2(\hat K)}\le Ch^{m_0+1}.
         \label{Numerator}
     \end{equation}
    By Assumption~\ref{appen:assumption}, there exists $m_*\ge2$ such that $\Vert\nabla^{m_*}\phi(x_K)\Vert\ge c_0$. Note that $\oplus_{m=2}^{m_0}\mathcal{B}_{m, \hat{x}_K}$ is direct, and $\hat{\pi}_{h, K}\in\oplus_{m=2}^{m_0}\mathcal{B}_{m, \hat{x}_K}$. We have the following estimates for sufficiently small $h$:
     \begin{equation}
         \Vert\hat{\pi}_{h, K}\Vert_{L^2(\hat K)}\ge C_eh^{m_*}\Vert\hat{b}_{m_*, K}\Vert_{L^2(\hat K)}\ge\frac{C_ec_1c_0}{m_0!}h^{m_0}.
         \label{Denominator}
     \end{equation}
     Combining \eqref{Numerator} and \eqref{Denominator}, we have
     \begin{equation}
         \frac{\Vert\hat{R}_{h, K}\Vert_{H^2(\hat K)}}{\Vert\hat{\pi}_{h, K}\Vert_{L^2(\hat K)}}\to0, \qquad h\to0.
         \label{limit}
     \end{equation}
     On the other hand, there exists $C_{m_0}>0$
     \begin{equation}
         C_{m_0} = \sup_{0\neq \hat \pi\in\oplus_{m=2}^{m_0}\mathcal{B}_{m, \hat{x}_K}}\frac{\Vert \hat{\pi}\Vert_{H^2(\hat K)}}{\Vert \hat \pi\Vert_{L^2(\hat K)}}<\infty.
         \label{Cm0}
     \end{equation}
     Normalize $\hat r_{h, K}$ by $\hat{\psi}_{h, K}:=\frac{\hat{r}_{h, K}}{\Vert\hat r_{h, K}\Vert_{L^2(\hat K)}}$. Then $
        \hat{V}_{h, K} = \mathbb P_1(\hat{K}) \oplus \text{span}\{\hat{L}_j\hat{\psi}_{h, K}\}_{j=1}^M=\hat{S}(\hat{\psi}_{h, K}).
    $
     From \eqref{Cm0}, the following estimate holds
     \begin{equation}
         \Vert\hat\psi_{h, K}\Vert_{H^2(\hat K)}=\frac{\Vert\hat\pi_{h, K} + \hat{R}_{h, K}\Vert_{H^2(\hat K)}}{\Vert\hat\pi_{h, K} + \hat{R}_{h, K}\Vert_{L^2(\hat K)}}\le\frac{C_{m_0}\Vert\hat\pi_{h, K}\Vert_{L^2(\hat K)} + \Vert\hat{R}_{h, K}\Vert_{H^2(\hat K)}}{\Vert\hat\pi_{h, K}\Vert_{L^2(\hat K)} - \Vert\hat{R}_{h, K}\Vert_{L^2(\hat K)}}
         \label{frac}
     \end{equation}
     Applying \eqref{limit} and \eqref{frac}, $\mathcal{F}:=\{\hat \psi_{h,K}\}$ is uniformly bounded in $H^2(\hat K)$. Since the embedding $H^2(\hat K)\hookrightarrow H^1(\hat K)$ is compact, $\mathcal{F}$ is therefore relatively compact in $H^1(\hat K)$.
    
    (Step 2) The best possible constant $C(h, K)$ is given by:
    \begin{equation}
        C(h, K) \le C(\hat{\psi})=\sup_{\hat v\in\hat{S}(\hat{\psi})\setminus\{0\}}\frac{\Vert\hat\nabla\hat v\Vert_{L^2(\hat K)}}{\Vert\hat v\Vert_{L^2(\hat K)}}=\sup_{\substack{\hat v \in \hat S(\hat{\psi}) \\ \Vert\hat{v}\Vert_{L^2(\hat K)}=1}}\Vert\hat{\nabla}\hat{v}\Vert_{L^2(\hat K)}.
    \end{equation}
    We show that $C(\hat \psi)$ is upper semi-continuous with respect to $\hat \psi$.
    Since $\hat \psi$ ranges over a compact set $\overline{\mathcal{F}}$,
    this implies that $C(h,K)$ is uniformly bounded. 
    
    Let $\hat{\psi}_n\to\hat{\psi}\in H^1(\hat K)$. Since $\hat S(\hat \psi_n)$ is finite-dimensional, the supremum defining $C(\hat \psi_n)$ is attained. Hence, for each $\hat \psi_n$ there exists $\hat v_n = \hat p_n + \hat \psi_n \hat q_n \in \hat S(\hat \psi_n)$ such that
    $
    \|\nabla \hat v_n\|_{L^2(\hat K)} = C(\hat \psi_n)$, $\|\hat v_n\|_{L^2(\hat K)} = 1$.
    Here, $\{p_n\}$ and $\{q_n\}$ are bounded. If not, we can find a subsequence such that $A_n := \Vert p_n\Vert_{L^2(\hat{K})} + \Vert q_n\Vert_{L^2(\hat{K})} \to \infty$. Let
$
\tilde{p}_n := \frac{p_n}{A_n}, \quad \tilde{q}_n := \frac{q_n}{A_n}$, which yields $\tilde{p}_n + \psi_n\tilde{q}_n = \frac{v_n}{A_n}.$
It then follows that $
\Vert\tilde{p}_n\Vert_{L^2(\hat{K})} + \Vert\tilde{q}_n\Vert_{L^2(\hat{K})} = 1$, and $\Vert\tilde{p}_n+\psi_n\tilde{q}_n\Vert_{L^2(\hat K)} \to 0.
$
We extract a convergent subsequence such that $\tilde{p}_n \to \tilde{p}$ and $\tilde{q}_n \to \tilde{q}$. Consequently,
$
\tilde{p} + \psi \tilde{q} = 0$ a.e. in $\hat{K}$.
Noticing that $\psi$ vanishes at each node of $\hat{K}$, we deduce that $\tilde{p} \equiv 0$, which implies $\tilde{q} \equiv 0$. This contradicts the condition $\Vert\tilde{p}\Vert_{L^2(\hat{K})} + \Vert\tilde{q}\Vert_{L^2(\hat{K})} = 1$.
    
We extract convergent subsequences $\hat{p}_{n_k}, \hat{q}_{n_k}\to\hat{p}, \hat{q}\in\mathbb P_1(\hat{K})$. Then,
    \begin{equation}
    C(\hat{\psi}_{n_k})=\Vert\hat{\nabla}\hat{v}_{n_k}\Vert_{L^2(\hat{K})}\to\Vert\hat{\nabla}\hat v\Vert_{L^2(\hat K)}\le C(\hat{\psi}),
    \end{equation}
    which implies that $\limsup_{n\to+\infty}C(\hat{\psi}_n)\le C(\hat{\psi})$.
    Therefore, $C(\hat\psi)$ is upper semi-continuous; \eqref{inverse} holds.
\end{proof}

To handle the flux jump term on element edges,
we introduce a further assumption that is slightly stronger than
Assumption~\ref{appen:assumption}.
\begin{assumption}\label{appen:edge assumption}Let $\mathcal{E}$ denote the set of direction vectors associated with the mesh $\mathcal{T}_h$. Let the enrichment function satisfy $\phi\in W^{m_0+1, \infty}$ for some $2\le m_0<\infty$. There exists a constant $c_0>0$ such that, $\forall x\in\bar\Omega$,
    \begin{equation}
        \forall \tau\in\mathcal{E}, \max_{2\le m\le m_0}\vert \nabla^m\phi(x)[\tau, \cdots, \tau]\vert\ge c_0>0,
        \label{assumption_t}
    \end{equation}
    where $\nabla^m\phi[\tau, \cdots, \tau]$ denotes the $m$th-order directional derivative in the direction $\tau$. 
\end{assumption}
\begin{remark}
        Assumption~\ref{appen:edge assumption} mainly excludes the degenerate case in which the enrichment function is locally close to a linear polynomial along the mesh edge directions. As an unfitted method, the discretization can be carried out on structured meshes. In this case, the set of mesh edge directions $\mathcal{E}$ is finite.
\end{remark}
As a consequence of Lemma \ref{appen:edge assumption}, we have the following inverse estimate.
\begin{lemma}
    Consider the enrichment function $\phi$ satisfy Assumption~\ref{appen:edge assumption}. Let $v_h \in V_h$, and denote by $w_h := [\partial_n v_h]$ the normal flux jump across an interior edge $e$. There exists a constant $C>0$ independent of $h$ such that
    \begin{equation}
        \Vert\nabla w_h\Vert_{L^2(e)}\le Ch^{-1}\Vert w_h\Vert_{L^2(e)}.
    \end{equation}
\end{lemma}
\begin{proof}
    Being similar to \eqref{flux space}, we have
        \begin{equation}
        w_h\in W_{h, e}:=\mathbb{P}_1(e)\oplus\text{span}\{r|_e\}.
    \end{equation}
    The rest of the proof then follows the same lines as that of Lemma~\ref{appendix: Inverse inequality}.
\end{proof}
\bibliographystyle{siamplain}
\bibliography{references}
\end{document}

%% file: ex_shared.tex
\usepackage{lipsum}
\usepackage{amsfonts}
\usepackage{graphicx}
\usepackage{epstopdf}
\usepackage{algorithmic}
\ifpdf
  \DeclareGraphicsExtensions{.eps,.pdf,.png,.jpg}
\else
  \DeclareGraphicsExtensions{.eps}
\fi

\newsiamremark{remark}{Remark}
\newsiamremark{hypothesis}{Hypothesis}
\crefname{hypothesis}{Hypothesis}{Hypotheses}
\newsiamthm{claim}{Claim}
\newsiamremark{fact}{Fact}
\crefname{fact}{Fact}{Facts}

\headers{Neural enrichment finite element method}{S. Guo, T. Richter}

\title{Neural enrichment finite element method: A hybrid method for problems with strong oscillations and interface problems\thanks{Submitted to the editors DATE.
\funding{The authors gratefully acknowledge the funding by the European Regional Development Fund (ERDF) within the programme Research and Innovation - Grant Number ZS/2023/12/182075. SG is also affiliated to the International Max Planck Research School for Systems and Process Engineering for a Sustainable Chemical Production (IMPRS SysProSus), Magdeburg, Germany.}}}

\author{Shihan Guo\thanks{Institute for Analysis and Numerics, Otto-von-Guericke-Universität Magdeburg, Magdeburg, 39106, Germany 
  (\email{shihan.guo@ovgu.de}, \email{thomas.richter@ovgu.de}).}
\and Thomas Richter\footnotemark[2]}

\usepackage{amsopn}

%% file: ex_article.bbl
\begin{thebibliography}{10}

\bibitem{aballay2025r}
{\sc D.~Aballay, F.~Fuentes, V.~Iligaray, {\'A}.~J. Omella, D.~Pardo, M.~A. S{\'a}nchez, I.~Tapia, and C.~Uriarte}, {\em An r-adaptive finite element method using neural networks for parametric self-adjoint elliptic problems}, J. Comput. Phys.,  (2025), p.~114447.

\bibitem{AlnaesEtal2015}
{\sc M.~S. Alnaes, J.~Blechta, J.~Hake, A.~Johansson, B.~Kehlet, A.~Logg, C.~Richardson, J.~Ring, M.~E. Rognes, and G.~N. Wells}, {\em The {FEniCS} project version 1.5}, Arch. Numer. Software, 3 (2015).

\bibitem{AlnaesEtal2014}
{\sc M.~S. Alnaes, A.~Logg, K.~B. {\O}lgaard, M.~E. Rognes, and G.~N. Wells}, {\em Unified form language: A domain-specific language for weak formulations of partial differential equations}, ACM Trans. Math. Software, 40 (2014).

\bibitem{babuvska1970finite}
{\sc I.~Babu{\v{s}}ka}, {\em The finite element method for elliptic equations with discontinuous coefficients}, Computing, 5 (1970), pp.~207--213.

\bibitem{babuvska2012stable}
{\sc I.~Babu{\v{s}}ka and U.~Banerjee}, {\em Stable generalized finite element method (sgfem)}, Comput. Methods Appl. Mech. Engrg., 201 (2012), pp.~91--111.

\bibitem{babuvska2017strongly}
{\sc I.~Babu{\v{s}}ka, U.~Banerjee, and K.~Kergrene}, {\em Strongly stable generalized finite element method: Application to interface problems}, Comput. Methods Appl. Mech. Engrg., 327 (2017), pp.~58--92.

\bibitem{baek2024n}
{\sc J.~Baek, Y.~Wang, and J.-S. Chen}, {\em N-adaptive ritz method: A neural network enriched partition of unity for boundary value problems}, Comput. Methods Appl. Mech. Engrg., 428 (2024), p.~117070.

\bibitem{chamoin2023introductory}
{\sc L.~Chamoin and F.~Legoll}, {\em An introductory review on a posteriori error estimation in finite element computations}, SIAM Review, 65 (2023), pp.~963--1028.

\bibitem{cui2022stable}
{\sc C.~Cui, Q.~Zhang, U.~Banerjee, and I.~Babu{\v{s}}ka}, {\em Stable generalized finite element method (sgfem) for three-dimensional crack problems}, Numer. Math., 152 (2022), pp.~475--509.

\bibitem{frei2014locally}
{\sc S.~Frei and T.~Richter}, {\em A locally modified parametric finite element method for interface problems}, SIAM J. Numer. Anal., 52 (2014), pp.~2315--2334.

\bibitem{fries2008corrected}
{\sc T.-P. Fries}, {\em A corrected xfem approximation without problems in blending elements}, Internat. J. Numer. Methods Engrg., 75 (2008), pp.~503--532.

\bibitem{fries2010extended}
{\sc T.-P. Fries and T.~Belytschko}, {\em The extended/generalized finite element method: an overview of the method and its applications}, Internat. J. Numer. Methods Engrg., 84 (2010), pp.~253--304.

\bibitem{gong2024improved}
{\sc W.~Gong, H.~Li, and Q.~Zhang}, {\em Improved enrichments and numerical integrations in sgfem for interface problems}, J. Comput. Appl. Math., 438 (2024), p.~115540.

\bibitem{hansbo2002unfitted}
{\sc A.~Hansbo and P.~Hansbo}, {\em An unfitted finite element method, based on nitsche’s method, for elliptic interface problems}, Comput. Methods Appl. Mech. Engrg., 191 (2002), pp.~5537--5552.

\bibitem{hansbo2014cut}
{\sc P.~Hansbo, M.~G. Larson, and S.~Zahedi}, {\em A cut finite element method for a stokes interface problem}, Appl. Numer. Math, 85 (2014), pp.~90--114.

\bibitem{hong2025fem}
{\sc Y.~Hong, W.~Zhang, L.~Zhao, and H.~Zheng}, {\em Fem-msfem hybrid method for the stokes-darcy model}, J. Comput. Phys., 532 (2025), p.~113952.

\bibitem{hou1999convergence}
{\sc T.~Hou, X.-H. Wu, and Z.~Cai}, {\em Convergence of a multiscale finite element method for elliptic problems with rapidly oscillating coefficients}, Math. Comput., 68 (1999), pp.~913--943.

\bibitem{hou1997multiscale}
{\sc T.~Y. Hou and X.-H. Wu}, {\em A multiscale finite element method for elliptic problems in composite materials and porous media}, J. Comput. Phys., 134 (1997), pp.~169--189.

\bibitem{jagtap2020adaptive}
{\sc A.~D. Jagtap, K.~Kawaguchi, and G.~E. Karniadakis}, {\em Adaptive activation functions accelerate convergence in deep and physics-informed neural networks}, J. Comput. Phys., 404 (2020), p.~109136.

\bibitem{Kapustsin2023}
{\sc U.~Kapustsin, U.~Kaya, and T.~Richter}, {\em A hybrid finite element/neural network solver and its application to the poisson problem}, 2023, \url{https://doi.org/10.1002/pamm.202300135}, \url{http://arxiv.org/abs/2307.00947}.

\bibitem{liu2024symmetric}
{\sc C.~Liu and B.~Liu}, {\em Symmetric and asymmetric gauss and gauss--lobatto quadrature rules for triangles and their applications to high-order finite element analyses}, J. Comput. Appl. Math., 437 (2024), p.~115451.

\bibitem{HartmannLessigMargenbergRichter2020}
{\sc N.~Margenberg, D.~Hartmann, C.~Lessig, and T.~Richter}, {\em A neural network multigrid solver for the navier-stokes equations}, Journal of Computational Physics, 460 (2022), p.~110983, \url{https://doi.org/10.1016/j.jcp.2022.110983}, \url{https://arxiv.org/abs/2008.11520}.

\bibitem{MargenbergJendersieLessigRichter2023}
{\sc N.~Margenberg, R.~Jendersie, C.~Lessig, and T.~Richter}, {\em Dnn-mg: A hybrid neural network/finite element method with applications to 3d simulations of the navier-stokes equations}, Computer Methods in Applied Mechanics and Engineering, 420 (2024), p.~116692, \url{https://doi.org/10.1016/j.cma.2023.116692}.

\bibitem{melenk1996partition}
{\sc J.~M. Melenk and I.~Babu{\v{s}}ka}, {\em The partition of unity finite element method: basic theory and applications}, Comput. Methods Appl. Mech. Engrg., 139 (1996), pp.~289--314.

\bibitem{raissi2019physics}
{\sc M.~Raissi, P.~Perdikaris, and G.~E. Karniadakis}, {\em Physics-informed neural networks: A deep learning framework for solving forward and inverse problems involving nonlinear partial differential equations}, J. Comput. Phys., 378 (2019), pp.~686--707.

\bibitem{strouboulis2001generalized}
{\sc T.~Strouboulis, K.~Copps, and I.~Babu{\v{s}}ka}, {\em The generalized finite element method}, Comput. Methods Appl. Mech. Engrg., 190 (2001), pp.~4081--4193.

\bibitem{verfurth1994posteriori}
{\sc R.~Verf{\"u}rth}, {\em A posteriori error estimation and adaptive mesh-refinement techniques}, J. Comput. Appl. Math., 50 (1994), pp.~67--83.

\bibitem{wang2025general}
{\sc D.~Wang, H.~Li, and Q.~Zhang}, {\em General enrichments of stable gfem for interface problems: Theory and extreme learning machine construction}, Appl. Numer. Math, 214 (2025), pp.~143--159.

\bibitem{wang2025neural}
{\sc Y.~Wang, Z.~Lin, and H.~Xie}, {\em Neural network element method for partial differential equations}, arXiv preprint arXiv:2504.16862,  (2025).

\bibitem{yu2018deep}
{\sc B.~Yu and W.~E}, {\em The deep ritz method: a deep learning-based numerical algorithm for solving variational problems}, Commun. Math. Stat., 6 (2018), pp.~1--12.

\bibitem{zang2020weak}
{\sc Y.~Zang, G.~Bao, X.~Ye, and H.~Zhou}, {\em Weak adversarial networks for high-dimensional partial differential equations}, J. Comput. Phys., 411 (2020), p.~109409.

\bibitem{zhang2019strongly}
{\sc Q.~Zhang, U.~Banerjee, and I.~Babu{\v{s}}ka}, {\em Strongly stable generalized finite element method (ssgfem) for a non-smooth interface problem}, Comput. Methods Appl. Mech. Engrg., 344 (2019), pp.~538--568.

\end{thebibliography}
